\documentclass[sn-mathphys,Numbered]{sn-jnl}% Math and Physical Sciences Reference Style
\usepackage{graphicx}%
\usepackage{multirow}%
\usepackage{amsmath,amssymb,amsfonts}%
\usepackage{amsthm}%
\usepackage{mathrsfs}%
\usepackage[title]{appendix}%
\usepackage{xcolor}%
\usepackage{textcomp}%
\usepackage{manyfoot}%
\usepackage{booktabs}%
\usepackage{algorithm}%
\usepackage{algorithmicx}%
\usepackage{algpseudocode}%
\usepackage{listings}%
\usepackage{float} %Forcing exact graphic placement - H
\usepackage{chngcntr} %Required package in order to use the macro.
\usepackage{caption}
\usepackage{subcaption}
\usepackage{fullpage}
\theoremstyle{thmstyleone}%
\newtheorem{theorem}{Theorem}%  meant for continuous numbers
\theoremstyle{thmstyletwo}%
\newtheorem{example}{Example}%
\newtheorem{remark}{Remark}%

\theoremstyle{thmstylethree}%
\newtheorem{definition}{Definition}%
\newtheorem{lemma}[theorem]{Lemma}
\begin{document}

\title[Split multivalued variational inequalities]{Levitin-Polyak well-posedness of split multivalued variational inequalities}

%===============================================================
\author*[1,2]{\fnm{Soumitra} \sur{Dey}}\email{deysoumitra2012@gmail.com}

\author[1]{\fnm{Simeon} \sur{Reich}}\email{sreich@technion.ac.il}

\affil*[1]{\orgdiv{Department of Mathematics}, \orgname{The Technion -- Israel Institute of Technology}, \orgaddress{ \city{Haifa}, \postcode{3200003}, \country{Israel}}}

\affil[2]{\orgdiv{Department of Mathematics}, \orgname{University of Haifa}, \orgaddress{ \city{Haifa}, \postcode{3498838}, \country{Israel}}}


\abstract{We introduce and study the split multivalued variational inequality problem (SMVIP) and the parametric SMVIP. We examine, in particular, Levitin-Polyak well-posedness of SMVIPs and parametric SMVIPs in Hilbert spaces. We provide several examples to illustrate our theoretical results. We also discuss several important special cases.}

\keywords{Multivalued variational inequalities, Split inverse problems, Well-posedness}


\pacs[2020 MSC classification]{49K40, 49J40, 90C31, 47H10, 47J20}

\maketitle

%=========================================================================================================================
\section{Introduction}
\label{Sec:1}
The concept of well-posedness is one of the most important and interesting topics in nonlinear analysis and optimization theory. It has been studied in connection with many problems, such as variational inequalities (VIs), equilibrium problems (EPs) and inverse problems (IPs). The notion of well-posedness for a minimization problem was first introduced and examined by Tykhonov \cite{ANTY1966} in metric spaces. Recall that a minimization problem is said to be well-posed (or correctly posed) if its solution set is a singleton and every minimizing sequence converges to this unique solution (see, for instance, \cite{MFUR1970}). In practice, the minimizing sequence generated by a numerical method may fail to lie inside the feasible set although the sequence of the distances of its elements to the feasible set does tend to zero. Such a sequence is called a generalized minimizing sequence for the minimization problem at hand. Motivated and inspired by this fact, Levitin and Polyak \cite{ESLE1966} modified the notion of Tykhonov well-posedness by replacing minimizing sequences by generalized minimizing sequence. On the other hand, since the requirement that the solution set be a singleton may be too restrictive, Furi and Vignoli \cite{MFUR1970} relaxed the above definition of well-posedness when applied to optimal control and mathematical programming problems. This encouraged researchers to study the notion of well-posedness of problems with multiple solutions. Motivated by this development, Zolezzi \cite{TZOL1995,TZOL1996} introduced the concept of parametric minimization problems and extended the concepts of well-posedness and well-posedness by perturbations to a more general setting. For further developments in the area of well-posedness and Levitin-Polyak (LP) well-posedness, see \cite{XXHU2006,ALDO1993} and references therein.

Let $\mathscr{H}$ be a real Hilbert space equipped with the inner product $\left\langle\cdot, \cdot\right\rangle$ and the norm $\|\cdot\|$, and let $C$ be a nonempty, closed and convex subset of $\mathscr{H}$. Let $f:\mathscr{H}\rightarrow\mathscr{H}$ be a mapping.
Then the classical {\em variational inequality} (VI) problem is defined as follows: find a point $x^*\in C$ such that
\begin{align}\label{VIP}
\left\langle f(x^*), x-x^*\right\rangle\geq 0 \; \; \forall  x\in C.\tag{VI}
\end{align}

Variational inequalities are a useful tool for solving many nonlinear analysis and optimization problems such as systems of linear and nonlinear equations and complementarity problem (see, for instance, \cite{DKIN1980}). The variational inequality problem \eqref{VIP} was first introduced and studied by Stampacchia \cite{GSTA1964} (in finite-dimensional Euclidean spaces). Thereafter, many researchers have paid attention to it and have established various analytical results (see, for example, \cite{PHAR1966,DKIN1980}). It is well known that minimization problems and variational inequality problems are closely related.
Therefore, Lucchetti and Patrone \cite{RLUC1981,RLUC1982} were the first to extend the well-posedness concept to VIs. Thereafter, the Levitin-Polyak well-posedness was studied by Hu et al. \cite{RHU2010}, as well as by Huang et al. \cite{XXHU2007,XXHU2009}. Later on, the concept of well-posedness by perturbations for mixed variational inequality problems was introduced by Feng et al. \cite{YPFA2010}.
On the other hand, many optimizers focused their work on numerical methods for solving variational inequalities and introduced and studied several numerical algorithms (see, for instance,  \cite{YCEN2011, YCEN22011, XJCA2014, QLDO2019, QLDO2018}) and references therein. Also, for various generalizations of \eqref{VIP}, we refer to \cite{SCFA1982,EBLU1994,RSAI1976,SDEY22023} and references therein.

One of the most important generalizations of variational inequality problem \eqref{VIP} is the mixed variational inequality problem (see \cite{KQWU2007} in Banach spaces). For various kinds of well-posedness of mixed variational inequality problems in different spaces, see \cite{YPFA2008,YPFA2010}. Later on, the \eqref{VIP} was generalized by many researchers to what is commonly known as the multivalued variational inequality problem (MVIP). The MVIP was first introduced by Browder \cite{FEBR1965} when studying maximal monotone operators and duality mappings in Banach spaces. Recently, the MVIP has been studied by Thang et al. \cite{TVTH2022}, where the authors have introduced proximal-type algorithms with self-adaptive step size and have established the weak convergence of the sequences generated by them. For the well-posedness of a class of MVIPs, we refer the interested reader to \cite{LCCE2008}. However, the well-posedness of the MVIP not yet been studied widely.

Censor et al. \cite[Section 2]{YCEN22012} introduced the general split inverse problem (SIP) in which there are given two vector spaces,
$X$ and $Y$, and a bounded linear operator $A:X \rightarrow Y$. In addition, two inverse problems are involved. The first one, denoted by IP$_{1}$, is
formulated in the space $X$ and the second one, denoted by IP$_{2}$, is formulated in the space $Y$. Given these data, the {\em Split Inverse Problem}
(SIP) is formulated as follows:
\begin{gather}\label{IP}
\text{Find a point }x^* \in X\text{ that solves IP}_{1} \tag{$\text{IP}_1$}\\
\text{such that}\notag\\
\text{the point }y^*=Ax^*\in Y\text{ solves IP}_{2}.\tag{$\text{IP}_2$}
\end{gather}

The SIP is quite general because it enables one to obtain various split problems by making different choices of IP$_{1}$ and IP$_{2}$. Important examples are the split variational inequality problem (SVIP) (which was introduced and studied by Censor et al. \cite{YCEN22012}), the split inclusion problem (which was introduced and studied by Moudafi \cite{AMOU2011}), the split convex feasibility problem \cite{YCEN1994} and the split equilibrium problem \cite{SDEY2023}. There are several approaches to solving many SIPs (see, for instance, \cite{YCEN22012,AMOU2011,DVHI2022} and references therein). Various kinds of well-posedness for different kinds of SIPs can be found in \cite{RHUA2016,SDEY2023}.

Motivated by the above split problems, we introduce in this paper a new class of split problems. Let $\mathscr{H}_1$ and $\mathscr{H}_2$ be two real Hilbert spaces,
 and $C$ and $Q$ be two nonempty, closed and convex subsets of $\mathscr{H}_1$ and $\mathscr{H}_2$, respectively.
 Let $f:\mathscr{H}_1\rightarrow\mathbb{R}$ and $g:\mathscr{H}_2\rightarrow\mathbb{R}$ be two functions. Let $\mathscr{B}_1:\mathscr{H}_1\rightarrow 2^{\mathscr{H}_1}$ and $\mathscr{B}_2:\mathscr{H}_2\rightarrow 2^{\mathscr{H}_2}$ be two strict multivalued mappings and let $A:\mathscr{H}_1\rightarrow\mathscr{H}_2$ be a bounded linear operator. With these data, the {\em split multivalued variational inequality problem} (SMVIP) is formulated as follows: find $(x^*,y^*)\in\mathscr{H}_1\times \mathscr{H}_2$ so that there exists $(u^*,v^*)\in\mathscr{B}_1(x^*)\times\mathscr{B}_2(y^*)$ such that
\begin{equation}\label{SMVI}
\begin{cases}
x^*\in C, y^*\in Q, y^*=Ax^*;\\
\left\langle u^*, x^*-x\right\rangle+f(x^*)-f(x)\leq 0 \; \; \forall x\in C;\\
\left\langle v^*, y^*-y\right\rangle+g(y^*)-g(y)\leq 0 \; \; \forall y\in Q.\tag{SMVIP}
\end{cases}
\end{equation}
The solution set of \eqref{SMVI} is denoted by $S$.

Our paper is organized as follows. In Section \ref{Sec:2} some useful definitions and results are collected. In Section \ref{Sec:2:2} we introduce the concepts of an approximating sequence and of a generalized approximating sequence for the \eqref{SMVI}. Also, in this section, we introduce several well-posedness concepts for the \eqref{SMVI}. In Section \ref{Sec:2:2:2} we present a metric characterization of LP well-posedness for the \eqref{SMVI} and also, in this section, we provide several examples to illustrate our theoretical results.
Next, in Section \ref{Sec:3}, we introduce the parametric \eqref{SMVI} and the concepts of an approximating sequence and a generalized approximating sequence for the \eqref{PSMVI}. Also in this section, we introduce various kinds of well-posedness for the the $\mathcal{PSMVIP}$. In Section \ref{Sec:4} we establish a metric characterization of (generalized) LP well-posedness for the $\mathcal{PSMVIP}$. In Section \ref{Sec:5} we discus some important special cases of the \eqref{SMVI}.

%=======================================================================================================
\section{Preliminaries}\label{Sec:2}
\noindent
In this section we shall collect some basic definitions and results, which are required for proving our main results.

Let $A$ and $B$ be two nonempty subsets of a real Hilbert space $\mathscr{H}$. The {\em Hausdorff metric} $H(\cdot, \cdot)$ between $A$ and $B$ is defined by
\begin{align*}
H(A,B) := \max\left\lbrace D(A,B), D(B,A)\right\rbrace,
\end{align*}
where $D(A,B) := \sup_{a\in A}d(a,B)$ and $d(a,B) := \inf_{b\in B}\|a-b\|$.
Let $\left\lbrace A_n\right\rbrace$ be a sequence of subsets of $\mathscr{H}$.
We say that the sequence $\left\lbrace A_n\right\rbrace$ converges to $A$ if $H(A_n,A)\rightarrow 0$.
It is not difficult to see that $\{D(A_n,A)\}$ $\rightarrow 0$ if and only if $\{d(a_n,A)\} \rightarrow 0$, uniformly for all selections $a_n\in A_n$.

The {\em diameter of a set} $A$ is defined by
\begin{align*}
\text{diam}(A) := \sup \left\lbrace \|x-y\|: x, y\in A\right\rbrace.
\end{align*}

The {\em Kuratowski measure of noncompactness} of a set $A$ is defined by
\begin{align}
\mu(A) := \inf\left\lbrace\epsilon>0: A\subset\cup_{i=1}^n A_i,  \text{diam}(A_i)<\epsilon,~ i=1,2,\cdots, n, ~n\in\mathbb{N} \right\rbrace. \nonumber.
\end{align}

\begin{lemma}
Let $A$ and $B$ be two nonempty subsets of a real Hilbert space $\mathscr{H}$. Then
\begin{align*}
\mu(A)\leq 2 H(A,B)+\mu(B).
\end{align*}
\end{lemma}

It is important to mention that a multivalued mapping $F:\mathscr{H}\rightarrow 2^\mathscr{H}$ is said to be strict if $F(x)$ is nonempty for all $x\in\mathscr{H}$.

\begin{definition}\cite{UMOS1967,LCCE2008}
A strict weakly compact-valued mapping $B:\mathscr{H}\rightarrow 2^\mathscr{H}$ is said to be $H$-continuous at a point $x\in \mathscr{H}$ if for any $\epsilon>0$, there exists a number $\delta>0$ such that for all $y\in\mathscr{H}$ with $\|x-y\|<\delta$, one has $H(B(x),B(y))<\epsilon$. If this mapping $B:\mathscr{H}\rightarrow 2^\mathscr{H}$ is $\mathscr{H}$-continuous at each $x\in\mathscr{H}$, then we say that it is $\mathscr{H}$-continuous.
\end{definition}

\begin{lemma}\label{Lem:weaklycompact}\cite{HHBA2011}
Let $C$ be a nonempty subset of $\mathscr{H}$. Then the followings statements are equivalent:
\begin{enumerate}
\item $C$ is weakly compact.
\item $C$ is weakly sequentially compact.
\item $C$ is weakly closed and bounded.
\end{enumerate}
\end{lemma}

\begin{lemma}\label{Lem:2}
Let $C$ be a nonempty convex subset of a real Hilbert space $\mathscr{H}$. Then for each $x\in\mathscr{H}$, the distance function $d(x,C)$ (defined above) is continuous and convex.
\end{lemma}

Next, we recall a result of Nadler's.

\begin{theorem}[Nadler, \cite{RLUC1981}]     \label{Nadler Theorem}
Let $(X, \|\cdot\|)$ be a normed vector space and let $CB(X)$ be the collection of all nonempty, closed and bounded subsets of $X$. If $U$ and $V$ belong to $CB(X)$, then for any $\epsilon>0$ and any point $x\in U$, there exists a point $y \in V$ such that $\|x-y\|\leq (1+\epsilon)H(U,V)$. In particular, if both $U$ and $V$ are compact, then there exists a point $y \in V$ such that $\|x-y\|\leq H(U,V)$.
\end{theorem}
%==================================================================================================
\section{Split Multivalued Variational Inequality Problem \eqref{SMVI}}
\label{Sec:2:2}

\begin{definition}
A sequence $\{(x_n,y_n)\}\subset \mathscr{H}_1\times\mathscr{H}_2$ is said to be an approximating sequence for \eqref{SMVI} if there exist a sequence $\{(u_n,v_n)\}\subset\mathscr{H}_1\times\mathscr{H}_2$ with $u_n\in\mathscr{B}_1(x_n)$ and $v_n\in\mathscr{B}_2(y_n)$, and a positive sequence $\{\epsilon_n\}$ of real numbers with $\epsilon_n\rightarrow 0$ such that
\begin{equation}\label{ASSMVI}
\begin{cases}
x_n\in C, y_n\in Q, \|y_n-Ax_n\|\leq \epsilon_n;\\
\left\langle u_n, x_n-x\right\rangle+f(x_n)-f(x)\leq\epsilon_n \; \; \forall x\in C;\\
\left\langle v_n, y_n-y\right\rangle+g(y_n)-g(y)\leq\epsilon_n \; \; \forall y\in Q.\tag{ASMVI}
\end{cases}
\end{equation}
\end{definition}

\begin{definition}
The \eqref{SMVI} is said to be well-posed if $S$ is a singleton and every approximating sequence for \eqref{SMVI} converges to the unique solution. The \eqref{SMVI} is said well-posed in the generalized sense if $S$ is nonempty and every approximating sequence for \eqref{SMVI} has a subsequence which converges to some element of $S$.
\end{definition}

\begin{definition}
A sequence $\{(x_n,y_n)\}\subset\mathscr{H}_1\times\mathscr{H}_2$ is said to be a generalized approximating sequence for \eqref{SMVI} if there exist a sequence $\{(u_n,v_n)\}\subset\mathscr{H}_1\times\mathscr{H}_2$ with $u_n\in\mathscr{B}_1(x_n)$ and $v_n\in\mathscr{B}_2(y_n)$, and a positive sequence $\{\epsilon_n\}$ of real numbers with $\epsilon_n\rightarrow 0$ such that
\begin{equation}\label{GSMVI}
\begin{cases}
d(x_n,C)\leq\epsilon_n, d(y_n,Q)\leq\epsilon_n, \|y_n-Ax_n\|\leq \epsilon_n;\\
\left\langle u_n, x_n-x\right\rangle+f(x_n)-f(x)\leq\epsilon_n \; \; \forall x\in C;\\
\left\langle v_n, y_n-y\right\rangle+g(y_n)-g(y)\leq\epsilon_n \; \; \forall y\in Q.\tag{GASMVI}
\end{cases}
\end{equation}
\end{definition}

\begin{definition}
The \eqref{SMVI} is said to be LP well-posed if $S$ is a singleton and each generalized approximating sequence for \eqref{SMVI} converges to the unique solution. The \eqref{SMVI} is said to be LP well-posed in the generalized sense if $S$ is nonempty and every generalized approximating sequence for \eqref{SMVI} has a subsequence which converges to some element of $S$.
\end{definition}

For each $\epsilon> 0$, consider the following approximate solution set for \eqref{SMVI}:
\begin{multline}
S(\epsilon) := \Big\{(z, w)\in\mathscr{H}_1\times\mathscr{H}_2:d(z,C)\leq\epsilon, d(w,Q)\leq\epsilon, \|w-Az\|\leq\epsilon; \nonumber\\
\text{ there exist } u\in\mathscr{B}_1(z), v\in\mathscr{B}_2(w) \text{ such that }\\
\left\langle u, z-x\right\rangle+f(z)-f(x)\leq\epsilon \; \; \forall x\in C;\\
\left\langle v, w-y\right\rangle+g(w)-g(y)\leq\epsilon \; \; \forall y\in Q\Big\}.
\end{multline}

\begin{remark}\label{rmk5}
It is clear that
\begin{enumerate}
\item $S=S(0)$ and $S\subset S(\epsilon)$ for each $\epsilon>0$.
\item  $S(\epsilon_1)\subset S(\epsilon_2)$ for any $0<\epsilon_1\leq\epsilon_2$.
\end{enumerate}
\end{remark}

\begin{lemma}
Let $\mathscr{H}_1$ and $\mathscr{H}_2$ be two real Hilbert spaces, and let $C$ and $Q$ be two nonempty, closed and convex subsets of $\mathscr{H}_1$ and $\mathscr{H}_2$, respectively. Let $f:\mathscr{H}_1\rightarrow\mathbb{R}$ and $g:\mathscr{H}_2\rightarrow\mathbb{R}$ be two continuous functions. Let $\mathscr{B}_1:\mathscr{H}_1\rightarrow 2^{\mathscr{H}_1}$ and $\mathscr{B}_2:\mathscr{H}_2\rightarrow 2^{\mathscr{H}_2}$ be two strict weakly compact-valued and $H$-continuous mappings. Then the solution set $S$ of \eqref{SMVI} is closed.
\end{lemma}

\begin{proof}
Let the sequence $\{(x_n,y_n)\}\subset S$ be such that $(x_n,y_n)\rightarrow (\tilde{x},\tilde{y})$. Then there exists a sequence $\{(u_n,v_n)\}\subset\mathscr{H}_1\times\mathscr{H}_2$ with $u_n\in\mathscr{B}_1(x_n)$, $v_n\in\mathscr{B}_2(y_n)$ such that
\begin{equation}\label{ASSMVI2}
\begin{cases}
x_n\in C, y_n\in Q, y_n=Ax_n;\\
\left\langle u_n, x_n-x\right\rangle+f(x_n)-f(x)\leq 0 \; \; \forall x\in C;\\
\left\langle v_n, y_n-y\right\rangle+g(y_n)-g(y)\leq 0\; \; \forall y\in Q.
\end{cases}
\end{equation}

Since $C$ and $Q$ are closed and $A$ is a bounded linear operator, $(\tilde{x},\tilde{y})\in C\times Q$ and $\tilde{y}=A\tilde{x}$. Again, since $\mathscr{B}_1(x_n)$ and $\mathscr{B}_1(\tilde{x})$ are nonempty and weakly compact, $\mathscr{B}_1(x_n), \mathscr{B}_1(\tilde{x})\in CB(\mathscr{H}_1)$ for all $n\in\mathbb{N}$. Therefore, by Theorem \ref{Nadler Theorem}, for each $u_n\in\mathscr{B}_1(x_n)$, there exists $\tilde{v}_n\in\mathscr{B}_1(\tilde{x})$ such that
\begin{align*}
\|u_n-\tilde{v}_n\|\leq (1+1/n)H(\mathscr{B}_1(x_n),\mathscr{B}_1(\tilde{x})).
\end{align*}

Taking the limit as $n\rightarrow\infty$ in the above inequality and using the $H$-continuity of $\mathscr{B}_1$, we get $u_n-\tilde{v}_n\rightarrow 0$. Since $\mathscr{B}_1(\tilde{x})$ is weakly compact, without loss of generality, we may assume that $\tilde{v}_n\rightharpoonup \tilde{u}\in\mathscr{B}_1(\tilde{x})$ and consequently, $u_n\rightharpoonup\tilde{u}\in\mathscr{B}_1(\tilde{x})$. Similarly, we can find $\tilde{v}$ such that $v_n\rightharpoonup\tilde{v}\in\mathscr{B}_2(\tilde{y})$.

Now, taking the limit in \eqref{ASSMVI2} and using our assumptions, we get
\begin{equation*}
\begin{cases}
\tilde{x}\in C, \tilde{y}\in Q, \tilde{y}=A\tilde{x};\\
\left\langle \tilde{u}, \tilde{x}-x\right\rangle+f(\tilde{x})-f(x)\leq 0 \; \; \forall x\in C;\\
\left\langle \tilde{v}, \tilde{y}-y\right\rangle+g(\tilde{y})-g(y)\leq 0 \; \; \forall y\in Q.
\end{cases}
\end{equation*}
Thus, $(\tilde{x},\tilde{y})\in S$ and hence $S$ is closed.
This completes the proof.
\end{proof}
%=================================================================================================
\section{Metric characterization of well-posedness for \eqref{SMVI}}
\label{Sec:2:2:2}

\begin{theorem}\label{Thm13:2}
Let $\mathscr{H}_1$ and $\mathscr{H}_2$ be two real Hilbert spaces, and let $C$ and $Q$ be two nonempty, closed and convex subsets of $\mathscr{H}_1$ and $\mathscr{H}_2$, respectively. Let $f:\mathscr{H}_1\rightarrow\mathbb{R}$ and $g:\mathscr{H}_2\rightarrow\mathbb{R}$ be two continuous functions. Let $\mathscr{B}_1:\mathscr{H}_1\rightarrow 2^{\mathscr{H}_1}$ and $\mathscr{B}_2:\mathscr{H}_2\rightarrow 2^{\mathscr{H}_2}$ be two strict weakly compact-valued and $H$-continuous mappings.
Let $A:\mathscr{H}_1\rightarrow\mathscr{H}_2$ be a bounded linear operator.
Then the \eqref{SMVI} is LP well-posed if and only if
\begin{align}\label{Thm13:con1:2}
S(\epsilon)\neq\emptyset \quad \forall\epsilon>0 \text{ and } \text{diam}(S(\epsilon))\rightarrow 0 \text{ as }(\epsilon)\rightarrow 0.
\end{align}
\end{theorem}

\begin{proof}
Suppose \eqref{SMVI} is LP well-posed. Then, by definition, $S$ is a singleton. Since $S \subset S(\epsilon)$ for $\epsilon>0$,
$S(\epsilon)\neq\emptyset$ for all $\epsilon>0$.

We claim that $\text{diam}(S(\epsilon))\rightarrow 0 \text{ as }\epsilon\rightarrow 0$.
Suppose to the contrary that $\text{diam}S(\epsilon)\nrightarrow 0$ as $\epsilon\rightarrow 0$.

Then there exist $\bar{\delta}>0, 0<\epsilon_n\rightarrow 0, (z_n,w_n)\in S(\epsilon_n), (z_n^\prime,w_n^\prime)\in S(\epsilon_n)$ such that
\begin{align}\label{ineq1:2}
\|(z_n,w_n)-(z_n^\prime,w_n^\prime)\|>\bar{\delta} \; \;  \forall n\in\mathbb{N}.
\end{align}

Since $(z_n,w_n)\in S(\epsilon_n), (z_n^\prime,w_n^\prime)\in S(\epsilon_n)$, there exist $u_n\in\mathscr{B}_1(z_n)$, $v_n\in \mathscr{B}_2(w_n)$ and $u_n^\prime\in\mathscr{B}_1(z_n^\prime)$, $v_n^\prime\in\mathscr{B}_2(w_n^\prime)$ such that
\begin{equation*}%\label{GPSMVI6:2}
\begin{cases}
d(z_n,C)\leq\epsilon_n, d(w_n,Q)\leq\epsilon_n, \|w_n-Az_n\|\leq \epsilon_n;\\
\left\langle u_n, z_n-x\right\rangle+f(z_n)-f(x)\leq\epsilon_n \; \; \forall x\in C;\\
\left\langle v_n, w_n-y\right\rangle+g(w_n)-g(y)\leq\epsilon_n \; \; \forall y\in Q
\end{cases}
\end{equation*}
and
\begin{equation*}%\label{GPSMVI7:2}
\begin{cases}
d(z_n^\prime,C)\leq\epsilon_n, d(w_n^\prime,Q)\leq\epsilon_n, \|w_n^\prime-Az_n^\prime\|\leq \epsilon_n;\\
\left\langle u_n^\prime, z_n^\prime-x\right\rangle+f(z_n^\prime)-f(x)\leq\epsilon_n \; \; \forall x\in C;\\
\left\langle v_n^\prime, w_n^\prime-y\right\rangle+g(w_n^\prime)-g(y)\leq\epsilon_n \; \; \forall y\in Q.
\end{cases}
\end{equation*}

This implies that $\{(z_n,w_n)\}$, $\{(z_n^\prime,w_n^\prime)\}$ are generalized approximating sequence for the \eqref{SMVI}. Therefore, by the definition of LP well-posedness,
the sequences $\{(z_n,w_n)\}$ and $\{(z_n^\prime,w_n^\prime)\}$ converge to the unique solution of \eqref{SMVI}, which contradicts \eqref{ineq1:2}.
Thus, our claim does hold after all.

Conversely, suppose that \eqref{Thm13:con1:2} holds. Since $S\subset S(\epsilon)$ for any $\epsilon>0$, $S$ has at most one element.

Let $\{(z_n,w_n)\}\subset\mathscr{H}_1\times\mathscr{H}_2$ be any generalized approximating sequence for \eqref{SMVI}. Then there exist a sequence $\{(u_n,v_n)\}\subset\mathscr{H}_1\times\mathscr{H}_2$ with $u_n\in\mathscr{B}_1(z_n)$, $v_n\in\mathscr{B}_2(w_n)$ and a positive sequence $\{\epsilon_n\}$ of real numbers with $\epsilon_n\rightarrow 0$ such that
\begin{equation}\label{AS1:2}
\begin{cases}
d(z_n,C)\leq\epsilon_n, d(w_n,Q)\leq\epsilon_n, \|w_n-Az_n\|\leq \epsilon_n;\\
\left\langle u_n, z_n-x\right\rangle+f(z_n)-f(x)\leq\epsilon_n \; \ \forall x\in C;\\
\left\langle v_n, w_n-y\right\rangle+g(w_n)-g(y)\leq\epsilon_n \; \; \forall y\in Q.
\end{cases}
\end{equation}

Thus $(z_n,w_n)\in S(\epsilon_n)$.
Since diam$(S(\epsilon))\rightarrow 0 \text{ as } \epsilon\rightarrow 0$, $\{(z_n,w_n)\}$ is a Cauchy sequence. Therefore, $\{(z_n,w_n)\}$ converges strongly to a point $(\tilde{z},\tilde{w})\in\mathscr{H}_1\times\mathscr{H}_2$.

Since $\mathscr{B}_1(z_n)$ and $\mathscr{B}_1(\tilde{z})$ are nonempty weakly compact, we have $\mathscr{B}_1(z_n), \mathscr{B}_1(\tilde{z})\in CB(\mathscr{H}_1)$ for all $n\in\mathbb{N}$. Therefore, by Theorem \ref{Nadler Theorem}, for each $u_n\in\mathscr{B}_1(z_n)$, there exists $\tilde{v}_n\in\mathscr{B}_1(\tilde{z})$ such that
\begin{align*}
\|u_n-\tilde{v}_n\|\leq (1+1/n) H(\mathscr{B}_1(z_n),\mathscr{B}_1(\tilde{z}))\rightarrow 0 \text{ as } n\rightarrow\infty.
\end{align*}

This implies that $u_n-\tilde{v}_n\rightarrow 0$ as $n\rightarrow\infty$. Since $\mathscr{B}_1(\tilde{z})$ is weakly compact, without loss of generality, we may assume that $\tilde{v}_n\rightharpoonup \tilde{u}\in\mathscr{B}_1(\tilde{z})$. This implies that $u_n\rightharpoonup\tilde{u}\in\mathscr{B}_1(\tilde{z})$ as $n\rightarrow\infty$. Similarly, using the assumptions on $\mathscr{B}_2$, we can find a point $\tilde{v}$ such that $v_n\rightharpoonup\tilde{v}\in\mathscr{B}_2(\tilde{w})$.

Taking $n\rightarrow\infty$ in \eqref{AS1:2} and using our assumptions, we obtain
\begin{equation*}%\label{GPSMVI8:2}
\begin{cases}
\tilde{z}\in C, \tilde{w}\in Q, \tilde{w}=A\tilde{z};\\
\left\langle \tilde{u}, \tilde{z}-x\right\rangle+\varphi(\tilde{z},p)-\varphi(x,p)\leq 0 \; \; \forall x\in C;\\
\left\langle \tilde{v}, \tilde{w}-y\right\rangle+\psi(\tilde{w},p)-\psi(y,p)\leq 0 \; \; \forall y\in Q.
\end{cases}
\end{equation*}

This implies that $(\tilde{z},\tilde{w})\in S$ and that every generalized approximate sequence for \eqref{SMVI} converges to $(\tilde{z},\tilde{w})$.
Since we already know that $S$ has at most one element, it follows that it is a singleton.
Consequently, the \eqref{SMVI} is  LP well-posed and the proof is complete.
\end{proof}

\begin{theorem}\label{Thm16:2}
Let $\mathscr{H}_1$ and $\mathscr{H}_2$ be two real Hilbert spaces, and let $C$ and $Q$ be two nonempty, closed and convex subsets of $\mathscr{H}_1$ and $\mathscr{H}_2$, respectively. Let $f:\mathscr{H}_1\rightarrow\mathbb{R}$ and $g:\mathscr{H}_2\rightarrow\mathbb{R}$ be two continuous functions. Let $\mathscr{B}_1:\mathscr{H}_1\rightarrow 2^{\mathscr{H}_1}$ and $\mathscr{B}_2:\mathscr{H}_2\rightarrow 2^{\mathscr{H}_2}$ be two strict weakly compact-valued and $H$-continuous mappings.
Let $A:\mathscr{H}_1\rightarrow\mathscr{H}_2$ be a bounded linear operator. Then the \eqref{SMVI} is LP well-posed in the generalized sense if and only if
\begin{align}\label{ineq13:2}
S(\epsilon)\neq\emptyset \; \;  \forall\epsilon>0 \text{ and } \mu(S(\epsilon))\rightarrow 0 \text{ as }\epsilon\rightarrow 0.
\end{align}
\end{theorem}

\begin{proof}
Suppose that \eqref{SMVI} is LP well-posed in the generalized sense. Then $S$ is nonempty and hence $S(\epsilon)\neq\emptyset$ for any $\epsilon>0$.
Now, we show that $S$ is compact. Indeed, let $\{(z_n,w_n)\}\subset S$. Then it is obvious that $\{(z_n,w_n)\}$ is a generalized approximating sequence for the \eqref{SMVI}. Since the \eqref{SMVI} is LP well-posed in the generalized sense, the sequence $\{(z_n,w_n)\}$ has a subsequence which converges strongly to a point in $S$.
Hence $S$ is compact. Next, we need to show that
\begin{align*}
\mu(S(\epsilon))\rightarrow 0 \text{ as }\epsilon\rightarrow 0.
\end{align*}

Observe that for every $\epsilon>0$,
\begin{align*}
H(S(\epsilon), S)=\max\{D(S(\epsilon), S), D(S, S(\epsilon))\}=D(S(\epsilon), S).
\end{align*}

Taking into account the compactness of $S$, we get
\begin{align*}
 \mu(S(\epsilon))\leq 2 H(S(\epsilon),S).
\end{align*}

Therefore, it is sufficient to show that
\begin{align}%\label{ineq12:2}
H(S(\epsilon), S)\rightarrow 0 \text{ as } \epsilon\rightarrow 0.\nonumber
\end{align}

Suppose to the contrary that this does not hold.
Then there exist a number $\tau>0$, a positive sequence $\{\epsilon_n\}$ of real numbers with $\epsilon_n\rightarrow 0$ as $n\rightarrow\infty$, and a sequence $\{(z_n,w_n)\}$ such that $(z_n,w_n) \in S(\epsilon_n)$ for each natural number $n$ and
\begin{align}\label{Thm2:con2:2}
(z_n,w_n)\notin S+\mathbb{B}(0,\tau) \; \; \forall n\in\mathbb{N}.
\end{align}

Since $(z_n,w_n) \in S(\epsilon_n)$, there exists a sequence $\{(u_n,v_n)\}\subset\mathscr{H}_1\times\mathscr{H}_2$ with $u_n\in\mathscr{B}_1(z_n)$, $v_n\in\mathscr{B}_2(w_n)$ such that
\begin{equation*}%\label{AS3:2}
\begin{cases}
d(z_n,C)\leq\epsilon_n, d(w_n,Q)\leq\epsilon_n, \|w_n-Az_n\|\leq \epsilon_n;\\
\left\langle u_n, z_n-x\right\rangle+f(z_n)-f(x)\leq\epsilon_n \; \; \forall x\in C;\\
\left\langle v_n, w_n-y\right\rangle+g(w_n)-g(y)\leq\epsilon_n \; \; \forall y\in Q.
\end{cases}
\end{equation*}

Therefore, it is evident that $\{(z_n,w_n)\}$ is a generalized approximating sequence for the \eqref{SMVI}. Since the \eqref{SMVI} is LP well-posed in the generalized sense,  $\{(z_n,w_n)\}$ has a subsequence which converges to a point in $S$. This, however, contradicts \eqref{Thm2:con2:2}.
Therefore,
\begin{align*}
H(S(\epsilon),S)\rightarrow 0 \text{ as } \epsilon\rightarrow 0,
\end{align*}
as claimed.

Conversely, suppose that
\begin{align}%\label{ineq14}
S(\epsilon)\neq\emptyset \; \;  \forall\epsilon>0 \text{ and } \mu(S(\epsilon))\rightarrow 0 \text{ as }\epsilon\rightarrow 0.\nonumber
\end{align}

By our assumptions. it is clear that $S(\epsilon)$ is closed for any $\epsilon>0$.
Also, observe that
\begin{align*}
S=\bigcap_{\epsilon>0}S(\epsilon).
\end{align*}

Since $\mu(S(\epsilon))\rightarrow 0 \text{ as } \epsilon\rightarrow 0$, using Theorem 1 in \cite{CHOR1985} (or \cite[p. 412]{KKUR1968}),
one can show that $S$ is nonempty and compact, and
\begin{align}\label{ineq:Ex}
H(S(\epsilon),S)\rightarrow 0\text{ as }\epsilon\rightarrow 0.
\end{align}

Let $\{(z_n,w_n)\}\subset\mathscr{H}_1\times\mathscr{H}_2$ be any generalized approximating sequence for the \eqref{SMVI}. Then there exist a sequence $\{(u_n,v_n)\}\subset\mathscr{H}_1\times\mathscr{H}_2$ with $u_n\in\mathscr{B}_1(z_n)$, $v_n\in\mathscr{B}_2(w_n)$ and a positive sequence $\{\epsilon_n\}$ of real numbers
such that $\epsilon_n\rightarrow 0$ and
\begin{equation*}%\label{AS4:2}
\begin{cases}
d(z_n,C)\leq\epsilon_n, d(w_n,Q)\leq\epsilon_n, \|w_n-Az_n\|\leq \epsilon_n;\\
\left\langle u_n, z_n-x\right\rangle+f(z_n)-f(x)\leq\epsilon_n \; \; \forall x\in C;\\
\left\langle v_n, w_n-y\right\rangle+g(w_n)-g(y)\leq\epsilon_n \; \; \forall y\in Q.
\end{cases}
\end{equation*}
Thus,  $(z_n,w_n)\in S(\epsilon_n)$.

Using condition (\ref{ineq:Ex}), we find that
\begin{align*}
d((z_n,w_n), S)&\leq D(S(\epsilon_n),S)\\&
=\max\{D(S(\epsilon_n),S), D(S,S(\epsilon_n))\}\\&
=H(S(\epsilon_n),S)\rightarrow 0 \text{ as } n\rightarrow\infty.
\end{align*}

Since $S$ is compact, the sequence $\{(z_n, w_n)\}$ has a subsequence which converges to some element of $S$.
Hence the \eqref{SMVI} is LP well-posed in the generalized sense, as asserted.
This completes the proof.
\end{proof}

\begin{theorem}\label{Thm17:2}
Let $\mathscr{H}_1$ and $\mathscr{H}_2$ be two finite-dimensional real Hilbert spaces, and let $C$ and $Q$ be two nonempty, closed and convex subsets of $\mathscr{H}_1$ and $\mathscr{H}_2$, respectively. Let $f:\mathscr{H}_1\rightarrow\mathbb{R}$ and $g:\mathscr{H}_2\rightarrow\mathbb{R}$ be two continuous functions.
Let $\mathscr{B}_1:\mathscr{H}_1\rightarrow 2^{\mathscr{H}_1}$ and $\mathscr{B}_2:\mathscr{H}_2\rightarrow 2^{\mathscr{H}_2}$ be two strict compact-valued and $H$-continuous mappings. Let $A:\mathscr{H}_1\rightarrow\mathscr{H}_2$ be a bounded linear operator. Suppose there exists $\epsilon>0$ such that $S(\epsilon)$ is nonempty and bounded.
Then the \eqref{SMVI} is LP well-posed if and only if it has a unique solution.
\end{theorem}

\begin{proof}
Since the necessity part of the theorem is obvious, we need to prove only the sufficiency part. To this end, suppose \eqref{SMVI} has a unique solution, say $(x^*,y^*)$.
Let $\{(z_n, w_n)\}$ be a generalized approximating sequence for the \eqref{SMVI}. Then there exist a sequence $\{(u_n,v_n)\}\subset\mathscr{H}_1\times\mathscr{H}_2$, with $u_n\in\mathscr{B}_1(z_n)$, $v_n\in\mathscr{B}_2(w_n)$, and a positive sequence $\{\epsilon_n\}$ of real numbers with $\epsilon_n\rightarrow 0$ such that
\begin{equation*}%\label{AS5:2}
\begin{cases}
d(z_n,C)\leq\epsilon_n, d(w_n,Q)\leq\epsilon_n, \|w_n-Az_n\|\leq \epsilon_n;\\
\left\langle u_n, z_n-x\right\rangle+f(z_n)-f(x)\leq\epsilon_n \; \; \forall x\in C;\\
\left\langle v_n, w_n-y\right\rangle+g(w_n)-g(y)\leq\epsilon_n \; \; \forall y\in Q.
\end{cases}
\end{equation*}
Thus $(z_n,w_n)\in S(\epsilon_n)$. Let $\epsilon>0$ be such that $S(\epsilon)$ is nonempty and bounded. Then there exists a natural number $n_0$ such that
\begin{align}%\label{ineq15:2}
(z_n,w_n)\in S(\epsilon_n)\subset S(\epsilon) \; \; \forall n\geq n_0.\nonumber
\end{align}
This implies that the sequence $\{(z_n,w_n)\}$ is bounded. Let $\{(z_{n_k},w_{n_k})\}$ be any subsequence of $\{(z_n,w_n)\}$. Since the sequence $\{(z_n,w_n)\}$ is bounded, $\{(z_{n_k},w_{n_k})\}$ has a convergent subsequence. Without loss of generality, we may assume that $(z_{n_k},w_{n_k})\rightarrow (\tilde{z},\tilde{w})$ as $k\rightarrow\infty$. Also, it is clear that $(z_{n_k},w_{n_k})\in S(\epsilon_{n_k})$. Therefore, there exist a sequence $\{(u_{n_k},v_{n_k})\}\subset\mathscr{H}_1\times\mathscr{H}_2$ with $u_{n_k}\in\mathscr{B}_1(z_{n_k})$, $v_{n_k}\in\mathscr{B}_2(w_{n_k})$ and a positive sequence $\{\epsilon_{n_k}\}$ of real numbers with $\epsilon_{n_k}\rightarrow 0$ such that
\begin{equation}\label{AS6:2}
\begin{cases}
d(z_{n_k},C)\leq\epsilon_{n_k}, d(w_{n_k},Q)\leq\epsilon_{n_k}, \|w_{n_k}-Az_{n_k}\|\leq \epsilon_{n_k};\\
\left\langle u_{n_k}, z_{n_k}-x\right\rangle+f(z_{n_k})-f(x)\leq\epsilon_{n_k} \; \; \forall x\in C;\\
\left\langle v_{n_k}, w_{n_k}-y\right\rangle+g(w_{n_k})-g(y)\leq\epsilon_{n_k} \; \forall y\in Q.
\end{cases}
\end{equation}

Since $\mathscr{B}_1(z_{n_k})$ and $\mathscr{B}_1(\tilde{z})$ are nonempty compact, $\mathscr{B}_1(z_{n_k}), \mathscr{B}_1(\tilde{z})\in CB(\mathscr{H}_1)$ for all $k\in\mathbb{N}$. Therefore, by Theorem \ref{Nadler Theorem}, for each $u_{n_k}\in\mathscr{B}_1(z_{n_k})$, there exists $\tilde{v}_{n_k}\in\mathscr{B}_1(\tilde{z})$ such that
\begin{align*}
\|u_{n_k}-\tilde{v}_{n_k}\|\leq (1+1/k) H(\mathscr{B}_1(z_{n_k}),\mathscr{B}_1(\tilde{z}))\rightarrow 0 \text{ as } k\rightarrow\infty.
\end{align*}

This implies that $u_{n_k}-\tilde{v}_{n_k}\rightarrow 0$ as $k\rightarrow\infty$. Since  $\mathscr{B}_1(\tilde{z})$ is compact, without loss of generality, we may assume that $\tilde{v}_{n_k}\rightarrow\tilde{u}\in\mathscr{B}_1(\tilde{z})$. This implies that $u_{n_k}\rightarrow \tilde{u}\in\mathscr{B}_1(\tilde{z})$ as $ k\rightarrow\infty$. Similarly, using the assumptions on $\mathscr{B}_2$ we can find $\tilde{v}$ such that $v_{n_k}\rightarrow\tilde{v}\in\mathscr{B}_2(\tilde{w})$.

Taking $k\rightarrow\infty$ in \eqref{AS6:2} and using our assumptions, we obtain
\begin{equation*}%\label{GPSMVI9:2}
\begin{cases}
\tilde{z}\in C, \tilde{w}\in Q, \tilde{w}=A\tilde{z};\\
\left\langle\tilde{u}, \tilde{z}-x\right\rangle+f(\tilde{z})-f(x)\leq 0 \; \; \forall x\in C;\\
\left\langle\tilde{v}, \tilde{w}-y\right\rangle+g(\tilde{w})-g(y)\leq 0 \; \; \forall y\in Q.
\end{cases}
\end{equation*}

This implies that $(\tilde{z},\tilde{w})\in S$. Since the \eqref{SMVI} has a unique solution, $(x^*,y^*)=(\tilde{z},\tilde{w})$. Therefore, every generalized approximate sequence for the \eqref{SMVI} converges to $(\tilde{z},\tilde{w})$.
Consequently, the \eqref{SMVI} is LP well-posed. This completes the proof.
\end{proof}

\begin{example}\label{Example:1}
Let $\mathscr{H}_1=\mathbb{R}=\mathscr{H}_2$ and let $C=[0,1]=Q$. Let $A:\mathbb{R}\rightarrow\mathbb{R}$ be the bounded linear operator defined by $A(x)=x$.
Define two functions, $f:\mathbb{R}\rightarrow\mathbb{R}$ and $g:\mathbb{R}\rightarrow\mathbb{R}$, by $f(x)=x^4$ and $g(y)=y^2$. It is clear that both $f$ and $g$ are continuous functions on $\mathbb{R}$. Consider the two strict multivalued mappings $\mathscr{B}_1:\mathbb{R}\rightarrow 2^\mathbb{R}$ and $\mathscr{B}_2:\mathbb{R}\rightarrow 2^\mathbb{R}$ defined by $\mathscr{B}_1(x)=\{x,0\}$ and $\mathscr{B}_2(y)=\{y,0\}$. It is clear that both $\mathscr{B}_1$ and $\mathscr{B}_2$ are compact-valued.

Now, we have
\begin{align*}
D(\mathscr{B}_1(x),\mathscr{B}_1(y))=\max\{d(x,\mathscr{B}_1(y)),d(0,\mathscr{B}_1(y))\}=d(x,\mathscr{B}_1(y))=\min\{\vert x\vert,\vert x-y\vert\}\leq \vert x-y\vert.
\end{align*}

Since a similar computation holds for $D(\mathscr{B}_1(y),\mathscr{B}_1(x))$,
it follows that
\begin{align*}
H(\mathscr{B}_1(x),\mathscr{B}_1(y))=\max\{D(\mathscr{B}_1(x),\mathscr{B}_1(y)),D(\mathscr{B}_1(y),\mathscr{B}_1(x))\}\leq \vert x-y\vert.
\end{align*}
This shows that $\mathscr{B}_1$ is $H$-continuous on $\mathbb{R}$.

Similarly, we can show that $\mathscr{B}_2$ is also $H$-continuous on $\mathbb{R}$.

In addition, it is not difficult to check that $S=\{(0,0)\}\subset\mathbb{R}^2$ is a singleton.

Now, for each $\epsilon> 0$, we have
\begin{multline}
S(\epsilon)= \Big\{(z, w)\in\mathbb{R}\times\mathbb{R}:d(z,[0,1])\leq\epsilon, d(w,[0,1])\leq\epsilon, \vert w-z\vert\leq\epsilon; \nonumber\\
\text{ there exist } u\in\mathscr{B}_1(z), v\in\mathscr{B}_2(w) \text{ such that }\\
\left\langle u, z-x\right\rangle+z^4-x^4\leq\epsilon,\forall x\in [0,1];\\
\left\langle v, w-y\right\rangle+w^2-y^2\leq\epsilon,\forall y\in [0,1]\Big\}.
\end{multline}
One can check that $S(\epsilon)\subset [-\epsilon,1+\epsilon]\times[-\epsilon,1+\epsilon]$ for each $\epsilon>0$. Therefore, $S(\epsilon)$ is a bounded subset of $\mathbb{R}^2$. Also, $S\subset S(\epsilon)$ for any $\epsilon>0$ and consequently, $S(\epsilon)\neq\emptyset$.
Therefore, the corresponding \eqref{SMVI} associated to these data is LP well-posed by Theorem \ref{Thm17:2}.
\end{example}

\begin{theorem}\label{Thm18:2}
Let $\mathscr{H}_1$ and $\mathscr{H}_2$ be two finite-dimensional Hilbert spaces, and let $C$ and $Q$ be two nonempty, closed and convex subsets of $\mathscr{H}_1$ and $\mathscr{H}_2$, respectively. Let $f:\mathscr{H}_1\rightarrow\mathbb{R}$ and $g:\mathscr{H}_2\rightarrow\mathbb{R}$ be two continuous functions. Let $\mathscr{B}_1:\mathscr{H}_1\rightarrow 2^{\mathscr{H}_1}$ and $\mathscr{B}_2:\mathscr{H}_2\rightarrow 2^{\mathscr{H}_2}$ be two strict compact-valued and  $H$-continuous mappings, and let $A:\mathscr{H}_1\rightarrow\mathscr{H}_2$ be a bounded linear operator. Suppose there exists $\epsilon>0$ such that $S(\epsilon)$ is nonempty and bounded. Then the \eqref{SMVI} is LP well-posed in the generalized sense if and only if the \eqref{SMVI} has nonempty solution set $S$.
\end{theorem}

\begin{proof}
The necessary part is obvious. For the sufficiency part, let the \eqref{SMVI} have a nonempty solution set $S$. Let $\{(z_n, w_n)\}$ be a generalized approximating sequence for the \eqref{SMVI}. Then there exist a sequence $\{(u_n,v_n)\}\subset\mathscr{H}_1\times\mathscr{H}_2$ with $u_n\in\mathscr{B}_1(z_n)$, $v_n\in\mathscr{B}_2(w_n)$, and a positive sequence $\{\epsilon_n\}$ of real numbers with $\epsilon_n\rightarrow 0$ such that
\begin{equation*}%\label{AS7:2}
\begin{cases}
d(z_n,C)\leq\epsilon_n, d(w_n,Q)\leq\epsilon_n, \|w_n-Az_n\|\leq \epsilon_n;\\
\left\langle u_n, z_n-x\right\rangle+f(z_n)-f(x)\leq\epsilon_n \; \; \forall x\in C;\\
\left\langle v_n, w_n-y\right\rangle+g(w_n)-g(y)\leq\epsilon_n \; \; \forall y\in Q.
\end{cases}
\end{equation*}
Thus $(z_n,w_n)\in S(\epsilon_n)$. Let $\epsilon>0$ be such that $S(\epsilon)$ is nonempty and bounded. Then there exists a natural number $n_0$ such that
\begin{align}%\label{ineq16:2}
(z_n,w_n)\in S(\epsilon_n)\subset S(\epsilon)\; \; \forall n\geq n_0.\nonumber
\end{align}
This implies that $\{(z_n,w_n)\}$ is bounded. Let $\{(z_{n_k},w_{n_k})\}$ be any convergent subsequence of $\{(z_n,w_n)\}$. Since $\{(z_n,w_n)\}$ is bounded, $\{(z_{n_k},w_{n_k})\}$ has a convergent subsequence. Without loss of generality, we may assume that $(z_{n_k},w_{n_k})\rightarrow (\tilde{z},\tilde{w})$ as $k\rightarrow\infty$. Also, it is clear that $(z_{n_k},w_{n_k})\in S(\epsilon_{n_k})$. Therefore, there exist a sequence $\{(u_{n_k},v_{n_k})\}\subset\mathscr{H}_1\times\mathscr{H}_2$ with $u_{n_k}\in\mathscr{B}_1(z_{n_k})$, $v_{n_k}\in\mathscr{B}_2(w_{n_k})$ and a positive sequence $\{\epsilon_{n_k}\}$ of real numbers with $\epsilon_{n_k}\rightarrow 0$ such that
\begin{equation}\label{AS8:2}
\begin{cases}
d(z_{n_k},C)\leq\epsilon_{n_k}, d(w_{n_k},Q)\leq\epsilon_{n_k}, \|w_{n_k}-Az_{n_k}\|\leq \epsilon_{n_k};\\
\left\langle u_{n_k}, z_{n_k}-x\right\rangle+f(z_{n_k})-f(x)\leq\epsilon_{n_k} \; \; \forall x\in C;\\
\left\langle v_{n_k}, w_{n_k}-y\right\rangle+g(w_{n_k})-g(y)\leq\epsilon_{n_k} \; \; \forall y\in Q.
\end{cases}
\end{equation}

Since $\mathscr{B}_1(z_{n_k})$ and $\mathscr{B}_1(\tilde{z})$ are nonempty compact, $\mathscr{B}_1(z_{n_k}), \mathscr{B}_1(\tilde{z})\in CB(\mathscr{H}_1)$ for all $k\in\mathbb{N}$. Therefore, by Theorem \ref{Nadler Theorem}, for each $u_{n_k}\in\mathscr{B}_1(z_{n_k})$, there exists $\tilde{v}_{n_k}\in\mathscr{B}_1(\tilde{z})$ such that
\begin{align*}
\|u_{n_k}-\tilde{v}_{n_k}\|\leq (1+1/k) H(\mathscr{B}_1(z_{n_k}),\mathscr{B}_1(\tilde{z}))\rightarrow 0 \text{ as } k\rightarrow\infty.
\end{align*}

This implies that $u_{n_k}-\tilde{v}_{n_k}\rightarrow 0$ as $k\rightarrow\infty$. Since $\mathscr{B}_1(\tilde{z})$ is compact, without loss of generality, we may assume that $\tilde{v}_{n_k}\rightarrow\tilde{u}\in\mathscr{B}_1(\tilde{z})$. This implies that $u_{n_k}\rightarrow\tilde{u}\in\mathscr{B}_1(\tilde{z})$ as $ k\rightarrow\infty$. Similarly, using the assumptions on $\mathscr{B}_2$, we can find $\tilde{v}$ such that $v_{n_k}\rightarrow\tilde{v}\in\mathscr{B}_2(\tilde{w})$.

Taking the limit as $k\rightarrow\infty$ in \eqref{AS8:2} and using our assumptions, we obtain
\begin{equation*}%\label{GPSMVI9:2}
\begin{cases}
\tilde{z}\in C, \tilde{w}\in Q, \tilde{w}=A\tilde{z};\\
\left\langle\tilde{u}, \tilde{z}-x\right\rangle+f(\tilde{z})-f(x)\leq 0 \; \; \forall x\in C;\\
\left\langle\tilde{v}, \tilde{w}-y\right\rangle+g(\tilde{w})-g(y)\leq 0 \; \; \forall y\in Q.
\end{cases}
\end{equation*}

This implies that $(\tilde{z},\tilde{w})\in S$. Therefore, every generalized approximate sequence for \eqref{SMVI} converges to $(\tilde{z},\tilde{w})$.
Consequently, the \eqref{SMVI} is LP well-posed in the generalized sense, as claimed. This completes the proof.
\end{proof}

\begin{example}\label{Example:2}
Let $\mathscr{H}_1=\mathbb{R}=\mathscr{H}_2$ and let $C=[-1,1]=Q$. Let $A:\mathbb{R}\rightarrow\mathbb{R}$ be the bounded linear operator defined by $A(x)=x$.
Define two functions, $f:\mathbb{R}\rightarrow\mathbb{R}$ and $g:\mathbb{R}\rightarrow\mathbb{R}$, by $f(x) = (x^2 - 1)^2$ and $g(y)=(y^4-1)^2$.
It is clear that both $f$ and $g$ are continuous functions on $\mathbb{R}$. Consider the two strict multivalued mappings $\mathscr{B}_1:\mathbb{R}\rightarrow 2^\mathbb{R}$
and $\mathscr{B}_2:\mathbb{R}\rightarrow 2^\mathbb{R}$ defined by $\mathscr{B}_(x)=\{x,0\}$ and $\mathscr{B}_2(y)=\{y,0\}$.
It is clear that both $\mathscr{B}_1$ and $\mathscr{B}_2$ are compact-valued.

Now, we have
\begin{align*}
D(\mathscr{B}_1(x),\mathscr{B}_1(y))=\max\{d(x,\mathscr{B}_1(y)),d(0,\mathscr{B}_1(y))\}=d(x,\mathscr{B}_1(y))=\min\{\vert x\vert,\vert x-y\vert\}\leq \vert x-y\vert.
\end{align*}

Since a similar computation holds for $D(\mathscr{B}_1(y),\mathscr{B}_1(x))$, it follows that
\begin{align*}
H(\mathscr{B}_1(x),\mathscr{B}_1(y))=\max\{D(\mathscr{B}_1(x),\mathscr{B}_1(y)),D(\mathscr{B}_1(y),\mathscr{B}_1(x))\}\leq \vert x-y\vert.
\end{align*}
This shows that $\mathscr{B}_1$ is $H$-continuous on $\mathbb{R}$.

Similarly, we can show that $\mathscr{B}_2$ is also $H$-continuous on $\mathbb{R}$.

In addition, it is not difficult to check that $(-1,-1), (1,1)\in S$.  Therefore, the set $S$ is nonempty.

Now, for each $\epsilon> 0$, we have
\begin{multline}
S(\epsilon)= \Big\{(z, w)\in\mathbb{R}\times\mathbb{R}:d(z,[-1,1])\leq\epsilon, d(w,[-1,1])\leq\epsilon, \vert w-z\vert\leq\epsilon; \nonumber\\
\text{ there exist } u\in\mathscr{B}_1(z), v\in\mathscr{B}_2(w) \text{ such that }\\
\left\langle u, z-x\right\rangle+(z^2-1)^2-(x^2-1)^2\leq\epsilon,\forall x\in [-1,1];\\
\left\langle v, w-y\right\rangle+(w^4-1)^2-(y^4-1)^2\leq\epsilon,\forall y\in [-1,1]\Big\}.
\end{multline}
One can check that $S(\epsilon)\subset [-1-\epsilon,1+\epsilon]\times[-1-\epsilon,1+\epsilon]$ for each $\epsilon>0$. Therefore, $S(\epsilon)$ is a bounded subset of $\mathbb{R}^2$. Also, $S\subset S(\epsilon)$ for any $\epsilon>0$ and consequently, $S(\epsilon)\neq\emptyset$. Therefore, the corresponding \eqref{SMVI} associated to these data is LP well-posed by Theorem \ref{Thm18:2}.
\end{example}

%==================================================================================================
%==================================================================================================
%==================================================================================================
\section{Parametric Split Multivalued Variational Inequality Problem (PSMVIP)}
\label{Sec:3}

Let $\mathscr{H}_1$ and $\mathscr{H}_2$ be two real Hilbert spaces, and let $C$ and $Q$ be two nonempty, closed and convex subsets of $\mathscr{H}_1$ and $\mathscr{H}_2$, respectively. Let $X$ be a Banach space, and let $\varphi:\mathscr{H}_1\times X\rightarrow\mathbb{R}$ and $\psi:\mathscr{H}_2\times X\rightarrow\mathbb{R}$ be two functions.
Let $\mathscr{B}_1:\mathscr{H}_1\times X\rightarrow 2^{\mathscr{H}_1}$ and $\mathscr{B}_2:\mathscr{H}_2\times X\rightarrow 2^{\mathscr{H}_2}$ be two strict multivalued mappings,
and $A:\mathscr{H}_1\rightarrow\mathscr{H}_2$ be a bounded linear operator. With these data, the {\em parametric split multivalued variational inequality problem} (PSMVIP) is formulated as follows: for each $p\in X$, find $(x^*,y^*)\in \mathscr{H}_1\times\mathscr{H}_2$ so that there exists $(u^*,v^*)\in\mathscr{B}_1(x^*,p)\times\mathscr{B}_2(y^*,p)$
such that
\begin{equation}\label{PSMVI}
\begin{cases}
x^*\in C, y^*\in Q, y^*=Ax^*;\\
\left\langle u^*, x^*-x\right\rangle+\varphi(x^*,p)-\varphi(x,p)\leq 0 \; \; \forall x\in C;\\
\left\langle v^*, y^*-y\right\rangle+\psi(y^*,p)-\psi(y,p)\leq 0 \; \; \forall y\in Q.\tag{PSMVIP$_p$}
\end{cases}
\end{equation}
The solution set of the \eqref{PSMVI} is denoted by $S_p$ for each fixed $p\in X$. Also, we denote by $\mathcal{PSMVIP}$ the family $\{$\ref{PSMVI}$: p\in X\}$ of problems.

\begin{definition}
Let $p\in X$ and $\{p_n\}\subset X$ with $p_n\rightarrow p$. A sequence $\{(x_n,y_n)\}\subset \mathscr{H}_1\times\mathscr{H}_2$ is said to be an approximating sequence for the  \eqref{PSMVI} corresponding to $\{p_n\}$ if there exist a sequence $\{(u_n,v_n)\}\subset\mathscr{H}_1\times\mathscr{H}_2$ with $u_n\in\mathscr{B}_1(x_n,p_n)$, $v_n\in\mathscr{B}_2(y_n,p_n)$ and a positive sequence $\{\epsilon_n\}$ of real numbers with $\epsilon_n\rightarrow 0$ such that
\begin{equation}\label{PSMVI2}
\begin{cases}
x_n\in C, y_n\in Q, \|y_n-Ax_n\|\leq \epsilon_n;\\
\left\langle u_n, x_n-x\right\rangle+\varphi(x_n,p_n)-\varphi(x,p_n)\leq\epsilon_n \; \; \forall x\in C;\\
\left\langle v_n, y_n-y\right\rangle+\psi(y_n,p_n)-\psi(y,p_n)\leq\epsilon_n \; \; \forall y\in Q.\tag{AS}
\end{cases}
\end{equation}
\end{definition}

\begin{definition}

The $\mathcal{PSMVIP}$ is said to be well-posed if for each $p\in X$, $S_p$ is a singleton and for each sequence $\{p_n\}\subset X$ with $p_n\rightarrow p$, every approximating sequence for \eqref{PSMVI} corresponding to $\{p_n\}$ converges to the unique solution. The $\mathcal{PSMVIP}$ is said to be well-posed in the generalized sense
if for each $p\in X$, $S_p$ is nonempty and for each sequence $\{p_n\}\subset X$ with $p_n\rightarrow p$, every approximating sequence for the \eqref{PSMVI} corresponding to $\{p_n\}$ has a subsequence which converges to some element of $S_p$.
\end{definition}

\begin{definition}
Let $p\in X$ and $\{p_n\}\subset X$ with $p_n\rightarrow p$. A sequence $\{(x_n,y_n)\}\subset\mathscr{H}_1\times\mathscr{H}_2$ is said to be a generalized approximating sequence for the \eqref{PSMVI} corresponding to $\{p_n\}$ if there exist a sequence $\{(u_n,v_n)\}\subset\mathscr{H}_1\times\mathscr{H}_2$ with $u_n\in\mathscr{B}_1(x_n,p_n)$, $v_n\in\mathscr{B}_2(y_n,p_n)$ and a positive sequence $\{\epsilon_n\}$ of real numbers with $\epsilon_n\rightarrow 0$ such that
\begin{equation}\label{GPSMVI}
\begin{cases}
d(x_n,C)\leq\epsilon_n, d(y_n,Q)\leq\epsilon_n, \|y_n-Ax_n\|\leq \epsilon_n;\\
\left\langle u_n, x_n-x\right\rangle+\varphi(x_n,p_n)-\varphi(x,p_n)\leq\epsilon_n \; \; \forall x\in C;\\
\left\langle v_n, y_n-y\right\rangle+\psi(y_n,p_n)-\psi(y,p_n)\leq\epsilon_n \; \; \forall y\in Q.\tag{GAS}
\end{cases}
\end{equation}
\end{definition}

\begin{definition}
The $\mathcal{PSMVIP}$ is said to be LP well-posed if for each $p\in X$, $S_p$ is a singleton and for each sequence $\{p_n\}\subset X$ with $p_n\rightarrow p$, every generalized approximating sequence for the \eqref{PSMVI} corresponding to $\{p_n\}$ converges to the unique solution. The $\mathcal{PSMVIP}$ is said to be LP well-posed in the generalized sense if for each $p\in X$, $S_p$ is nonempty and for each sequence $\{p_n\}\subset X$ with $p_n\rightarrow p$, every generalized approximating sequence for the \eqref{PSMVI} corresponding to $\{p_n\}$ has a subsequence which converges to some element of $S_p$.
\end{definition}

For each $p\in X$ and $\delta, \epsilon> 0$, consider the following approximate solution set for the \eqref{PSMVI}:
\begin{multline}
S_p(\delta,\epsilon) := \bigcup_{q\in\mathbb{B}(p, \delta)}\Big\{(z, w)\in\mathscr{H}_1\times\mathscr{H}_2:d(z,C)\leq\epsilon, d(w,Q)\leq\epsilon, \|w-Az\|\leq\epsilon; \nonumber\\
\text{ there exist } u\in\mathscr{B}_1(z,q), v\in\mathscr{B}_2(w,q) \text{ such that }\\
\left\langle u, z-x\right\rangle+\varphi(z,q)-\varphi(x,q)\leq\epsilon \; \; \forall x\in C;\\
\left\langle v, w-y\right\rangle+\psi(w,q)-\psi(y,q)\leq\epsilon \; \; \forall y\in Q\Big\},
\end{multline}
where $\mathbb{B}(p,\delta)$ is the closed ball centered at $p$ with  radius $\delta$.

\begin{remark}\label{rmk4}
It is clear that
\begin{enumerate}
\item $S_p=S_p(0,0)$ and $S_p\subset S_p(\delta,\epsilon)$ for each $\epsilon,\delta>0$.
\item  $S_p(\delta_1,\epsilon_1)\subset S_p(\delta_2,\epsilon_2)$ for any $0<\delta_1\leq\delta_2$ and $0<\epsilon_1\leq\epsilon_2$.
\end{enumerate}
\end{remark}

%========================================================================================
\section{Metric characterization of well-posedness for $\mathcal{PSMVIP}$}
\label{Sec:4}

\begin{theorem}\label{Thm11}
Let $\mathscr{H}_1$ and $\mathscr{H}_2$ be two real Hilbert spaces and, let $C$ and $Q$ be two nonempty, closed and convex subsets of $\mathscr{H}_1$ and $\mathscr{H}_2$, respectively. Let $X$ be a finite-dimensional Banach space, and let $\varphi:\mathscr{H}_1\times X\rightarrow\mathbb{R}$ and $\psi:\mathscr{H}_2\times X\rightarrow\mathbb{R}$ be two continuous functions. Let $\mathscr{B}_1:\mathscr{H}_1\times X\rightarrow 2^{\mathscr{H}_1}$ and $\mathscr{B}_2:\mathscr{H}_2\times X\rightarrow 2^{\mathscr{H}_2}$ be two strict weakly compact-valued and $H$-continuous mappings, and let $A:\mathscr{H}_1\rightarrow\mathscr{H}_2$ be a bounded linear operator. Then the approximate solution set $S_p(\delta,\epsilon)$ of \eqref{PSMVI} is closed for any $\delta,\epsilon>0$.
\end{theorem}

\begin{proof}
Let $\{(z_n,w_n)\}\subset S_p(\delta,\epsilon)$ be any sequence such that $(z_n,w_n)\rightarrow (\tilde{z},\tilde{w})$. Since $\{(z_n,w_n)\}\subset S_p(\delta,\epsilon)$, there exist $q_n\subset\mathbb{B}(p,\delta)$ and a sequence $\{(u_n,v_n)\}\subset\mathscr{H}_1\times\mathscr{H}_2$ with $u_n\in\mathscr{B}_1(z_n,q_n)$, $v_n\in\mathscr{B}_2(w_n,q_n)$ such that
\begin{equation} \label{GPSMVI3}
\begin{cases}
d(z_n,C)\leq\epsilon, d(w_n,Q)\leq\epsilon, \|w_n-Az_n\|\leq \epsilon;\\
\left\langle u_n, z_n-x\right\rangle+\varphi(z_n,q_n)-\varphi(x,q_n)\leq\epsilon \; \; \forall x\in C;\\
\left\langle v_n, w_n-y\right\rangle+\psi(w_n,q_n)-\psi(y,q_n)\leq\epsilon \; \; \forall y\in Q.
\end{cases}
\end{equation}

Since $X$ is finite-dimensional, without loss of generality, we may assume that $q_n\rightarrow\bar{q}\in\mathbb{B}(p,\delta)$.

Since $\mathscr{B}_1(z_n,q_n)$ and $\mathscr{B}_1(\tilde{z},\bar{q})$ are nonempty weakly compact, $\mathscr{B}_1(z_n,q_n), \mathscr{B}_1(\tilde{z},\bar{q})\in CB(\mathscr{H}_1)$ for all $n\in\mathbb{N}$. Therefore, by Theorem \ref{Nadler Theorem}, for each $u_n\in\mathscr{B}_1(z_n,q_n)$, there exists $\tilde{v}_n\in\mathscr{B}_1(\tilde{z},\bar{q})$ such that
\begin{align*}
\|u_n-\tilde{v}_n\|\leq (1+1/n) H(\mathscr{B}_1(z_n,q_n),\mathscr{B}_1(\tilde{z},\bar{q}))\rightarrow 0 \text{ as } n\rightarrow\infty.
\end{align*}

This implies that $u_n-\tilde{v}_n\rightarrow 0$ as $ n\rightarrow\infty$. Since $\mathscr{B}_1(z^*,\bar{q})$ is weakly compact, without loss of generality, we may assume that $\tilde{v}_n\rightharpoonup\tilde{u}\in\mathscr{B}_1(z^*,\bar{q})$. This implies that  $u_n\rightharpoonup\tilde{u}\in\mathscr{B}_1(z^*,\bar{q})$ as $ n\rightarrow\infty$. Similarly, using the assumptions on $\mathscr{B}_2$, we can find $\tilde{v}$ such that $v_n\rightharpoonup\tilde{v}\in\mathscr{B}_2(w^*,\bar{q})$. Taking $n\rightarrow\infty$ in \eqref{GPSMVI3} and using our assumptions, we obtain that
\begin{equation*}%\label{GPSMVI4}
\begin{cases}
d(\tilde{z},C)\leq\epsilon, d(\tilde{w},Q)\leq\epsilon, \|\tilde{w}-A\tilde{z}\|\leq \epsilon;\\
\left\langle\tilde{u}, \tilde{z}-x\right\rangle+\varphi(\tilde{z},\bar{q})-\varphi(x,\bar{q})\leq\epsilon \; \; \forall x\in C;\\
\left\langle\tilde{v}, \tilde{w}-y\right\rangle+\psi(\tilde{w},\bar{q})-\psi(y,\bar{q})\leq\epsilon \; \; \forall y\in Q.
\end{cases}
\end{equation*}

This implies that $(\tilde{z},\tilde{w})\in S_p(\delta,\epsilon)$.
\end{proof}

\begin{theorem}\label{Thm12}
Let $\mathscr{H}_1$ and $\mathscr{H}_2$ be two real Hilbert spaces, and let $C$ and $Q$ be two nonempty, closed and convex subsets of $\mathscr{H}_1$ and $\mathscr{H}_2$, respectively. Let $X$ be a Banach space and let $\varphi:\mathscr{H}_1\times X\rightarrow\mathbb{R}$ and $\psi:\mathscr{H}_2\times X\rightarrow\mathbb{R}$ be two continuous functions. Let $\mathscr{B}_1:\mathscr{H}_1\times X\rightarrow 2^{\mathscr{H}_1}$ and $\mathscr{B}_2:\mathscr{H}_2\times X\rightarrow 2^{\mathscr{H}_2}$ be two strict  weakly compact-valued and $H$-continuous mappings, and let $A:\mathscr{H}_1\rightarrow\mathscr{H}_2$ is a bounded linear operator. Then, for each $p \in X$, we have
\begin{align*}
S_p=\bigcap_{\delta>0,\epsilon>0}S_p(\delta,\epsilon).
\end{align*}
\end{theorem}

\begin{proof}
It is clear that $S_p \subset S_p(\delta,\epsilon)$ for any $\delta,\epsilon>0$ and any $p \in X$. Consequently, $S_p\subset \bigcap_{\delta>0,\epsilon>0}S_p(\delta,\epsilon)$ for each $p\in X$.

Now, we claim that $\bigcap_{\delta>0,\epsilon>0}S_p(\delta,\epsilon) \subset S_p$ for any $p\in X$.
Indeed, fix $p \in X$ and let $(z,w)\in \bigcap_{\delta>0,\epsilon>0}S_p(\delta,\epsilon)$ be arbitrary.
Then, without loss of generality, there exist sequences $\delta_n>0$, $\epsilon_n>0$ with $\delta_n\rightarrow 0$ and $\epsilon_n\rightarrow 0$ as $n\rightarrow\infty$ such that $(z,w)\in S_p(\delta_n,\epsilon_n)$.
Thus there exist $q_n\in B(p,\delta_n)$ and $\{(u_n,v_n)\}\subset\mathscr{H}_1\times\mathscr{H}_2$ with $u_n\in\mathscr{B}_1(z,q_n)$, $v_n\in\mathscr{B}_2(w,q_n)$ such that
\begin{equation}\label{GPSMVI4}
\begin{cases}
d(z,C)\leq\epsilon_n, d(w,Q)\leq\epsilon_n, \|w-Az\|\leq \epsilon_n;\\
\left\langle u_n, z-x\right\rangle+\varphi(z,q_n)-\varphi(x,q_n)\leq\epsilon_n \; \; \forall x\in C;\\
\left\langle v_n, w-y\right\rangle+\psi(w,q_n)-\psi(y,q_n)\leq\epsilon_n \; \; \forall y\in Q.
\end{cases}
\end{equation}

Since $q_n\in\mathbb{B}(p,\delta_n)$ and $\delta_n\rightarrow 0$ as $n\rightarrow\infty$, $q_n\rightarrow p$ as $n\rightarrow\infty$.

Again, since $\mathscr{B}_1(z,q_n)$ and $\mathscr{B}_1(z,p)$ are nonempty weakly compact, $\mathscr{B}_1(z,q_n), \mathscr{B}_1(z,p)\in CB(\mathscr{H}_1)$ for all $n\in\mathbb{N}$. Therefore, by Theorem \ref{Nadler Theorem}, for each $u_n\in\mathscr{B}_1(z,q_n)$, there exists $\tilde{v}_n\in\mathscr{B}_1(z,p)$ such that
\begin{align*}
\|u_n-\tilde{v}_n\|\leq (1+1/n) H(\mathscr{B}_1(z,q_n),\mathscr{B}_1(z,p))\rightarrow 0 \text{ as } n\rightarrow\infty.
\end{align*}

This implies that $u_n-\tilde{v}_n\rightarrow 0$ as $ n\rightarrow\infty$. Since $\mathscr{B}_1(z,p)$ is weakly compact, without loss of generality, we may assume that $\tilde{v}_n\rightharpoonup u\in\mathscr{B}_1(z,p)$. This implies that  $u_n\rightharpoonup u\in\mathscr{B}_1(z,p)$ as $ n\rightarrow\infty$. Similarly, using the assumptions on $\mathscr{B}_2$ we can find $v$ such that $v_n\rightharpoonup v\in\mathscr{B}_2(w,p)$.

Taking the limit as $ n\rightarrow\infty$ in \eqref{GPSMVI4} and using our assumptions, we obtain
\begin{equation*} %\label{GPSMVI5}
\begin{cases}
z\in C, w\in Q, w=Az;\\
\left\langle u, z-x\right\rangle+\varphi(z,p)-\varphi(x,p)\leq 0 \; \; \forall x\in C;\\
\left\langle v, w-y\right\rangle+\psi(w,p)-\psi(y,p)\leq 0 \; \; \forall y\in Q.
\end{cases}
\end{equation*}

This implies that $(z,w)\in S_p$ for any $p\in X$. This completes the proof.
\end{proof}

\begin{theorem}\label{Thm13}
Let $\mathscr{H}_1$ and $\mathscr{H}_2$ be two real Hilbert spaces, and let $C$ and $Q$ be two nonempty, closed and convex subsets of $\mathscr{H}_1$ and $\mathscr{H}_2$, respectively. Let $X$ be a Banach space and let $\varphi:\mathscr{H}_1\times X\rightarrow\mathbb{R}$ and $\psi:\mathscr{H}_2\times X\rightarrow\mathbb{R}$ be two functions. Let $\mathscr{B}_1:\mathscr{H}_1\times X\rightarrow 2^{\mathscr{H}_1}$ and $\mathscr{B}_2:\mathscr{H}_2\times X\rightarrow 2^{\mathscr{H}_2}$ be two strict multivalued mappings, and let $A:\mathscr{H}_1\rightarrow\mathscr{H}_2$ be a bounded linear operator.
Then the $\mathcal{PSMVIP}$ is LP well-posed if and only if for each $p\in X$, the solution set $S_p$ of the \eqref{PSMVI} is nonempty and
\begin{align}\label{Thm13:con1}
\text{diam}(S_p(\delta,\epsilon))\rightarrow 0 \text{ as }(\delta,\epsilon)\rightarrow (0,0).
\end{align}
\end{theorem}

\begin{proof}
Suppose the $\mathcal{PSMVIP}$ is LP well-posed. Then, by definition, $S_p$ is nonempty for each $p\in X$.

We claim that $\text{diam}(S_p(\delta,\epsilon))\rightarrow 0 \text{ as }(\delta,\epsilon)\rightarrow (0,0)$.
Suppose to the contrary that $\text{diam}S_p(\delta,\epsilon)\nrightarrow 0$ as $(\delta,\epsilon)\rightarrow (0,0)$.

Then there exist $\bar{\delta}>0, 0<\delta_n\rightarrow 0, 0<\epsilon_n\rightarrow 0, (z_n,w_n)\in S_p(\delta_n,\epsilon_n), (z_n^\prime,w_n^\prime)\in S_p(\delta_n,\epsilon_n)$ such that
\begin{align}\label{ineq1}
\|(z_n,w_n)-(z_n^\prime,w_n^\prime)\|>\bar{\delta} \; \;  \forall n\in\mathbb{N}.
\end{align}

Since $(z_n,w_n)\in S_p(\delta_n,\epsilon_n), (z_n^\prime,w_n^\prime)\in S_p(\delta_n,\epsilon_n)$, there exist $q_n,q_n^\prime\in\mathbb{B}(p,\delta_n)$ and  $u_n\in\mathscr{B}_1(z_n,q_n)$, $v_n\in \mathscr{B}_2(w_n,q_n)$ and $u_n^\prime\in\mathscr{B}_1(z_n^\prime,q_n^\prime)$, $v_n^\prime\in\mathscr{B}_2(w_n^\prime,q_n^\prime)$
such that
\begin{equation*} %\label{GPSMVI6}
\begin{cases}
d(z_n,C)\leq\epsilon_n, d(w_n,Q)\leq\epsilon_n, \|w_n-Az_n\|\leq \epsilon_n;\\
\left\langle u_n, z_n-x\right\rangle+\varphi(z_n,q_n)-\varphi(x,q_n)\leq\epsilon_n \; \; \forall x\in C;\\
\left\langle v_n, w_n-y\right\rangle+\psi(w_n,q_n)-\psi(y,q_n)\leq\epsilon_n \; \; \forall y\in Q
\end{cases}
\end{equation*}
and
\begin{equation*} %\label{GPSMVI7}
\begin{cases}
d(z_n^\prime,C)\leq\epsilon_n, d(w_n^\prime,Q)\leq\epsilon_n, \|w_n^\prime-Az_n^\prime\|\leq \epsilon_n;\\
\left\langle u_n^\prime, z_n^\prime-x\right\rangle+\varphi(z_n^\prime,q_n^\prime)-\varphi(x,q_n^\prime)\leq\epsilon_n \; \; \forall x\in C;\\
\left\langle v_n^\prime, w_n^\prime-y\right\rangle+\psi(w_n^\prime,q_n^\prime)-\psi(y,q_n^\prime)\leq\epsilon_n \; \; \forall y\in Q.
\end{cases}
\end{equation*}

Since $q_n,q_n^\prime\in\mathbb{B}(p,\delta_n)$ and $\delta_n\rightarrow 0$ as $n\rightarrow\infty$, $q_n\rightarrow p$ and $q_n^\prime\rightarrow p$ as $n\rightarrow\infty$. This implies that $\{(z_n,w_n)\}$ and $\{(z_n^\prime,w_n^\prime)\}$ are generalized approximating sequences for the \eqref{PSMVI} corresponding to $\{q_n\}$ and $\{q_n^\prime\}$, respectively. Therefore, by the definition of LP well-posedness of the $\mathcal{PSMVIP}$, the sequences $\{(z_n,w_n)\}$ and $\{(z_n^\prime,w_n^\prime)\}$ converge to the unique solution of the \eqref{PSMVI}, which contradicts \eqref{ineq1}.

Conversely, suppose that $S_p$ is nonempty for each $p\in X$ and \eqref{Thm13:con1} holds. Since $S_p\subset S_p(\delta,\epsilon)$ for any $\delta,\epsilon>0$ and $p\in X$,
the set $S_p$ is a singleton for each $p\in X$.

Let $\{p_n\}$ be any sequence in $X$ such that $p_n\rightarrow p$ as $n\rightarrow\infty$. Let $\{(z_n,w_n)\}\subset\mathscr{H}_1\times\mathscr{H}_2$ be any generalized approximating sequence for the \eqref{PSMVI} corresponding to $\{p_n\}$. Then there exist a sequence $\{(u_n,v_n)\}\subset\mathscr{H}_1\times\mathscr{H}_2$
with $u_n\in\mathscr{B}_1(z_n,p_n)$, $v_n\in\mathscr{B}_2(w_n,p_n)$ and a positive sequence $\{\epsilon_n\}$ of real numbers with $\epsilon_n\rightarrow 0$ such that
\begin{equation*} %\label{AS1}
\begin{cases}
d(z_n,C)\leq\epsilon_n, d(w_n,Q)\leq\epsilon_n, \|w_n-Az_n\|\leq \epsilon_n;\\
\left\langle u_n, z_n-x\right\rangle+\varphi(z_n,p_n)-\varphi(x,p_n)\leq\epsilon_n \; \; \forall x\in C;\\
\left\langle v_n, w_n-y\right\rangle+\psi(w_n,p_n)-\psi(y,p_n)\leq\epsilon_n \; \; \forall y\in Q.
\end{cases}
\end{equation*}

Let $\delta_n=\|p_n-p\|$. Then $p_n\in\mathbb{B}(p,\delta_n)$ and $\delta_n\rightarrow 0$ as $n\rightarrow\infty$. Thus, $(z_n,w_n)\in S_p(\delta_n, \epsilon_n)$.
Let $(\tilde{z},\tilde{w})\in S_p$. Then $(\tilde{z},\tilde{w})\in S_p(\delta_n,\epsilon_n)$.

Now, we have
\begin{align}
\|(z_n,w_n)-(\tilde{z},\tilde{w})\|\leq diam(S_p(\delta_n,\epsilon_n))\rightarrow 0 \text{ as } n\rightarrow 0.\nonumber
\end{align}

Since $\{(z_n,w_n)\}$ is arbitrary, we see that every approximating sequence for the \eqref{PSMVI} corresponding to $\{p_n\}$ converges to the unique solution.
Hence the $\mathcal{PSMVIP}$ is LP well-posed and the proof is complete.
\end{proof}

\begin{theorem}\label{Thm14}
Let $\mathscr{H}_1$ and $\mathscr{H}_2$ be two real Hilbert spaces, and let $C$ and $Q$ be two nonempty, closed and convex subsets of $\mathscr{H}_1$ and $\mathscr{H}_2$, respectively. Let $X$ be a Banach space and let $\varphi:\mathscr{H}_1\times X\rightarrow\mathbb{R}$ and $\psi:\mathscr{H}_2\times X\rightarrow\mathbb{R}$ be two continuous functions. Let $\mathscr{B}_1:\mathscr{H}_1\times X\rightarrow 2^{\mathscr{H}_1}$ and $\mathscr{B}_2:\mathscr{H}_2\times X\rightarrow 2^{\mathscr{H}_2}$ be two strict weakly compact-valued and $H$-continuous mappings, and let $A:\mathscr{H}_1\rightarrow\mathscr{H}_2$ be a bounded linear operator.
Then the $\mathcal{PSMVIP}$ is LP well-posed if and only if for every $p\in X$, we have
\begin{align}\label{ineq2}
S_p(\delta,\epsilon)\neq\emptyset \; \; \forall\delta,\epsilon>0 \text{ and } \text{diam}(S_p(\delta,\epsilon))\rightarrow 0 \text{ as }(\delta,\epsilon)\rightarrow (0,0).
\end{align}
\end{theorem}

\begin{proof}
The necessity part is clear from Theorem \ref{Thm13}. So, we only need to prove the sufficiency part. To this end, suppose that condition \eqref{ineq2} holds.
Since $S_p\subset S_p(\delta,\epsilon)$ for any $\delta, \epsilon>0$, the \eqref{PSMVI} has at most one solution for each $p\in X$.

Let $\{p_n\}$ be any sequence in $X$ such that $p_n\rightarrow p\in X$ as $n\rightarrow\infty$. Let $\{(z_n,w_n)\}\subset\mathscr{H}_1\times\mathscr{H}_2$ be any generalized approximating sequence for the \eqref{PSMVI} corresponding to $\{p_n\}$. Then there exist a sequence $\{(u_n,v_n)\}\subset\mathscr{H}_1\times\mathscr{H}_2$ with $u_n\in\mathscr{B}_1(z_n,p_n)$, $v_n\in\mathscr{B}_2(w_n,p_n)$ and a positive sequence $\{\epsilon_n\}$ of real numbers with $\epsilon_n\rightarrow 0$ such that
\begin{equation}\label{AS2}
\begin{cases}
d(z_n,C)\leq\epsilon_n, d(w_n,Q)\leq\epsilon_n, \|w_n-Az_n\|\leq \epsilon_n;\\
\left\langle u_n, z_n-x\right\rangle+\varphi(z_n,p_n)-\varphi(x,p_n)\leq\epsilon_n \; \; \forall x\in C;\\
\left\langle v_n, w_n-y\right\rangle+\psi(w_n,p_n)-\psi(y,p_n)\leq\epsilon_n \; \; \forall y\in Q.
\end{cases}
\end{equation}

Let $\delta_n=\|p_n-p\|$. Then $p_n\in\mathbb{B}(p,\delta_n)$ and $\delta_n\rightarrow 0$ as $n\rightarrow\infty$. Thus $(z_n,w_n)\in S_p(\delta_n, \epsilon_n)$. Since diam$(S_p(\delta,\epsilon))\rightarrow 0 \text{ as } (\delta,\epsilon)\rightarrow 0$, $\{(z_n,w_n)\}$ is a Cauchy sequence. Therefore, the sequence $\{(z_n,w_n)\}$ converges strongly to a point $(\tilde{z},\tilde{w})\in\mathscr{H}_1\times\mathscr{H}_2$.

Since $\mathscr{B}_1(z_n,p_n)$ and $\mathscr{B}_1(\tilde{z},p)$ are nonempty weakly compact, $\mathscr{B}_1(z_n,p_n)$, $\mathscr{B}_1(\tilde{z},p)\in CB(\mathscr{H}_1)$ for all $n\in\mathbb{N}$. Therefore, by Theorem \ref{Nadler Theorem}, for each $u_n\in\mathscr{B}_1(z_n,p_n)$, there exists $\tilde{v}_n\in\mathscr{B}_1(\tilde{z},p)$ such that
\begin{align*}
\|u_n-\tilde{v}_n\|\leq (1+1/n) H(\mathscr{B}_1(z_n,p_n),\mathscr{B}_1(\tilde{z},p))\rightarrow 0 \text{ as } n\rightarrow\infty.
\end{align*}

This implies that $u_n-\tilde{v}_n\rightarrow 0$ as $ n\rightarrow\infty$. Since $\mathscr{B}_1(\tilde{z},p)$ is weakly compact, without loss of generality, we may assume that $\tilde{v}_n\rightharpoonup\tilde{u}\in\mathscr{B}_1(\tilde{z},p)$. This implies that  $u_n\rightharpoonup\tilde{u}\in\mathscr{B}_1(\tilde{z},p)$ as $ n\rightarrow\infty$. Similarly, using the assumptions on $\mathscr{B}_2$, we can find $\tilde{v}$ such that $v_n\rightharpoonup\tilde{v}\in\mathscr{B}_2(\tilde{w},p)$.

Taking the limit as $n\rightarrow\infty$ in \eqref{AS2} and using our assumptions, we obtain
\begin{equation*}%\label{GPSMVI8}
\begin{cases}
\tilde{z}\in C, \tilde{w}\in Q, \tilde{w}=A\tilde{z};\\
\left\langle\tilde{u}, \tilde{z}-x\right\rangle+\varphi(\tilde{z},p)-\varphi(x,p)\leq 0 \; \; \forall x\in C;\\
\left\langle\tilde{v}, \tilde{w}-y\right\rangle+\psi(\tilde{w},p)-\psi(y,p)\leq 0 \; \; \forall y\in Q.
\end{cases}
\end{equation*}

This implies that $(\tilde{z},\tilde{w})\in S_p$ and every generalized approximate sequence for the \eqref{PSMVI} corresponding to $\{p_n\}$ converges to $(\tilde{z},\tilde{w})$.  This completes the proof.

\end{proof}

\begin{theorem}\label{Thm15}
Let $\mathscr{H}_1$ and $\mathscr{H}_2$ be two real Hilbert spaces, and let $C$ and $Q$ be two nonempty, closed and convex subsets of $\mathscr{H}_1$ and $\mathscr{H}_2$, respectively. Let $X$ be a Banach space, and let $\varphi:\mathscr{H}_1\times X\rightarrow\mathbb{R}$ and $\psi:\mathscr{H}_2\times X\rightarrow\mathbb{R}$ be two functions. Let $\mathscr{B}_1:\mathscr{H}_1\times X\rightarrow 2^{\mathscr{H}_1}$ and $\mathscr{B}_2:\mathscr{H}_2\times X\rightarrow 2^{\mathscr{H}_2}$ be two strict multivalued mappings, and let $A:\mathscr{H}_1\rightarrow\mathscr{H}_2$ be a bounded linear operator. Then the $\mathcal{PSMVIP}$ is LP well-posed in the generalized sense if and only if for each $p\in X$, the solution set $S_p$ of the \eqref{PSMVI} is nonempty compact and
\begin{align}\label{ineq11}
H(S_p(\delta,\epsilon), S_p)\rightarrow 0 \text{ as } (\delta,\epsilon)\rightarrow 0.
\end{align}
\end{theorem}

\begin{proof}
Suppose that $\mathcal{PSMVIP}$ is LP well-posed in the generalized sense. Then $S_p$ is nonempty for each $p\in X$. To show the compactness of $S_p$, fix $p\in X$ and let $\{(z_n,w_n)\}\subset S_p$. Then it is obvious that $\{(z_n,w_n)\}$ is a generalized approximating sequence for the \eqref{PSMVI} corresponding to $\{p_n\}=\{p\}$.
Since the $\mathcal{PSMVIP}$ is LP well-posed in the generalized sense, the sequence $\{(z_n,w_n)\}$ has a subsequence which converges strongly to a point in $S_p$.
Hence $S_p$ is compact for each $p\in X$, as claimed.

Now, we claim that

\begin{align}%\label{ineq12}
H(S_p(\delta, \epsilon), S_p)\rightarrow 0 \text{ as } (\delta,\epsilon)\rightarrow 0.\nonumber
\end{align}

Suppose to the contrary that this does not hold.
Then there exists a number $\tau>0$, positive sequences $\{\delta_n\}$ and $\{\epsilon_n\}$ with $\delta_n\rightarrow 0$, $\epsilon_n\rightarrow 0$ as $n\rightarrow\infty$, and a sequence $\{(z_n,w_n)\}$ such that $(z_n,w_n) \in S_p(\delta_n,\epsilon_n)$ for each natural number $n$ and
\begin{align}\label{Thm2:con2}
(z_n,w_n)\notin S_p+\mathbb{B}(0,\tau) \; \; \forall n\in\mathbb{N}.
\end{align}

Since $(z_n,w_n) \in S_p(\delta_n,\epsilon_n)$, there exist sequences $\{q_n\}\subset\mathbb{B}(p,\delta_n)$ and $\{(u_n,v_n)\}\subset\mathscr{H}_1\times\mathscr{H}_2$ with $u_n\in\mathscr{B}_1(z_n,q_n)$, $v_n\in\mathscr{B}_2(w_n,q_n)$ such that
\begin{equation*}%\label{AS3}
\begin{cases}
d(z_n,C)\leq\epsilon_n, d(w_n,Q)\leq\epsilon_n, \|w_n-Az_n\|\leq \epsilon_n;\\
\left\langle u_n, z_n-x\right\rangle+\varphi(z_n,q_n)-\varphi(x,q_n)\leq\epsilon_n \; \; \forall x\in C;\\
\left\langle v_n, w_n-y\right\rangle+\psi(w_n,q_n)-\psi(y,q_n)\leq\epsilon_n \; \; \forall y\in Q.
\end{cases}
\end{equation*}

Therefore, it is evident that $\{(z_n,w_n)\}$ is a generalized approximating sequence for the \eqref{PSMVI} corresponding to $\{q_n\}$. Since $\mathcal{PSMVIP}$ is LP well-posed in the generalized sense, the sequence $\{(z_n,w_n)\}$ has a subsequence which converges to a point in $S_p$, which contradicts \eqref{Thm2:con2}.

Conversely, for each $p\in X$, the solution set $S_p$ of the \eqref{PSMVI} is nonempty and compact, and \eqref{ineq11} holds.
Let $\{p_n\}$ be any sequence in $X$ such that $p_n\rightarrow p\in X$ as $n\rightarrow\infty$.
Let $\{(z_n,w_n)\}\subset\mathscr{H}_1\times\mathscr{H}_2$ be any generalized approximating sequence for the \eqref{PSMVI} corresponding to $\{p_n\}$. Then there exist a sequence $\{(u_n,v_n)\}\subset\mathscr{H}_1\times\mathscr{H}_2$ with $u_n\in\mathscr{B}_1(z_n,p_n)$, $v_n\in\mathscr{B}_2(w_n,p_n)$ and a positive sequence $\{\epsilon_n\}$ of real numbers with $\epsilon_n\rightarrow 0$ such that
\begin{equation*}%\label{AS4}
\begin{cases}
d(z_n,C)\leq\epsilon_n, d(w_n,Q)\leq\epsilon_n, \|w_n-Az_n\|\leq \epsilon_n;\\
\left\langle u_n, z_n-x\right\rangle+\varphi(z_n,p_n)-\varphi(x,p_n)\leq\epsilon_n \; \; \forall x\in C;\\
\left\langle v_n, w_n-y\right\rangle+\psi(w_n,p_n)-\psi(y,p_n)\leq\epsilon_n \; \; \forall y\in Q.
\end{cases}
\end{equation*}
Let $\delta_n=\|p_n-p\|$. Then $p_n\in\mathbb{B}(p,\delta_n)$ and $\delta_n\rightarrow 0$ as $n\rightarrow\infty$. Thus $(z_n,w_n)\in S_p(\delta_n, \epsilon_n)$.

Using condition (\ref{ineq11}), we get
\begin{align*}
d((z_n,w_n), S_p)&\leq D(S_p(\delta_n,\epsilon_n),S_p)\\&
=\max\{D(S_p(\delta_n,\epsilon_n),S_p), D(S_p,S_p(\delta_n,\epsilon_n))\}\\&
=H(S_p(\delta_n,\epsilon_n),S_p)\rightarrow 0 \text{ as } n\rightarrow\infty.
\end{align*}

Since $S_p$ is compact, the sequence $\{(z_n, w_n)\}$ has a subsequence which converges to some element of $S_p$.
Hence $\mathcal{PSMVIP}$ is LP well-posed in the generalized sense.
\end{proof}

\begin{theorem}\label{Thm16}
Let $\mathscr{H}_1$ and $\mathscr{H}_2$ be two real Hilbert spaces, and let $C$ and $Q$ be two nonempty, closed and convex subsets of $\mathscr{H}_1$ and $\mathscr{H}_2$, respectively. Let $X$ be a finite-dimensional Banach space, and let $\varphi:\mathscr{H}_1\times X\rightarrow\mathbb{R}$ and $\psi:\mathscr{H}_2\times X\rightarrow\mathbb{R}$ be two continuous functions. Let $\mathscr{B}_1:\mathscr{H}_1\times X\rightarrow 2^{\mathscr{H}_1}$ and $\mathscr{B}_2:\mathscr{H}_2\times X\rightarrow 2^{\mathscr{H}_2}$ be two strict weakly compact-valued and $H$-continuous mappings, and let $A:\mathscr{H}_1\rightarrow\mathscr{H}_2$ be a bounded linear operator.
Then the $\mathcal{PSMVIP}$ is LP well-posed in the generalized sense if and only if for each $p\in X$, we have
\begin{align}\label{ineq13}
S_p(\delta,\epsilon)\neq\emptyset \; \; \forall\delta,\epsilon>0 \text{ and } \mu(S_p(\delta,\epsilon))\rightarrow 0 \text{ as } (\delta,\epsilon)\rightarrow 0.
\end{align}
\end{theorem}

\begin{proof}
Suppose $\mathcal{PSMVIP}$ is LP well-posed in the generalized sense. Then by Theorem \ref{Thm15}, the set $S_p$ is nonempty and compact for each $p\in X$, and we have
\begin{align*}
H(S_p(\delta,\epsilon),S_p)\rightarrow 0 \text{ as } (\delta,\epsilon)\rightarrow 0.
\end{align*}

Since $S_p\subset S_p(\delta, \epsilon)$ for any $\delta>0$, $\epsilon>0$, the set $S_p(\delta, \epsilon)$ is nonempty for each $p\in X$ and for all $\delta>0$, $\epsilon>0$.

Now we claim that $\mu(S_p(\delta,\epsilon))\rightarrow 0 \text{ as } (\delta,\epsilon)\rightarrow 0$. Indeed, for any $\delta,\epsilon>0$, we have
\begin{align*}
 \mu(S_p(\delta,\epsilon))\leq 2 H(S_p(\delta,\epsilon),S_p)+\mu(S_p)=2 H(S_p(\delta,\epsilon),S_p)\rightarrow 0 \text{ as } (\delta,\epsilon)\rightarrow 0.
\end{align*}

Conversely, suppose that \eqref{ineq13} holds for every $p \in X$.

By Theorem \ref{Thm11} and our assumptions, it is clear that $S_p(\delta,\epsilon)$ is closed for any $\delta,\epsilon>0$.
Also, we observe that (see, Theorem \ref{Thm12})
\begin{align*}
S_p=\bigcap_{\delta>0,\epsilon>0}S_p(\delta,\epsilon)
\end{align*}
for each $p\in X$.

Since $\mu(S_p(\delta,\epsilon))\rightarrow 0 \text{ as } (\delta,\epsilon)\rightarrow 0$ for each $p\in X$,
using Theorem 1 in \cite{CHOR1985} (see also \cite[p. 412]{KKUR1968}), one can show that for each $p\in X$, the set $S_p$ is nonempty and compact, and we have
\begin{align*}
H(S_p(\delta,\epsilon),S_p)\rightarrow 0\text{ as }(\delta,\epsilon)\rightarrow 0.
\end{align*}
Therefore, in view of Theorem \ref{Thm15}, the $\mathcal{PSMVIP}$ is LP well-posed in the generalized sense.
This completes the proof.
\end{proof}

\begin{theorem}\label{Thm17}
Let $\mathscr{H}_1$ and $\mathscr{H}_2$ be two finite-dimensional real Hilbert spaces, and let $C$ and $Q$ be two nonempty, closed and convex subsets of $\mathscr{H}_1$ and $\mathscr{H}_2$, respectively. Let $X$ be a Banach space, and let $\varphi:\mathscr{H}_1\times X\rightarrow\mathbb{R}$ and $\psi:\mathscr{H}_2\times X\rightarrow\mathbb{R}$ be two continuous functions. Let $\mathscr{B}_1:\mathscr{H}_1\times X\rightarrow 2^{\mathscr{H}_1}$ and $\mathscr{B}_2:\mathscr{H}_2\times X\rightarrow 2^{\mathscr{H}_2}$ be two strict compact-valued and $H$-continuous mappings, and let $A:\mathscr{H}_1\rightarrow\mathscr{H}_2$ be a bounded linear operator. Suppose that for each $p\in X$, there exists $\epsilon>0$ such that $S_p(\epsilon,\epsilon)$ is nonempty and bounded. Then the $\mathcal{PSMVIP}$ is LP well-posed if and only if for each $p\in X$, the \eqref{PSMVI} has a unique solution.
\end{theorem}

\begin{proof}
Since the necessity part of the theorem is obvious, we only need to prove the sufficiency part. Suppose that, for each $p\in X$, the \eqref{PSMVI} has a unique solution,
say $(x^*,y^*)$. Let $\{p_n\}$ be any sequence in $X$ such that $p_n\rightarrow p\in X$ as $n\rightarrow\infty$. Let $\{(z_n, w_n)\}$ be a generalized approximating sequence for the \eqref{PSMVI} corresponding to $\{p_n\}$. Then there exist a sequence $\{(u_n,v_n)\}\subset\mathscr{H}_1\times\mathscr{H}_2$
with $u_n\in\mathscr{B}_1(z_n,p_n)$, $v_n\in\mathscr{B}_2(w_n,p_n)$ and a positive sequence $\{\epsilon_n\}$ of real numbers with $\epsilon_n\rightarrow 0$ such that
\begin{equation*}%\label{AS5}
\begin{cases}
d(z_n,C)\leq\epsilon_n, d(w_n,Q)\leq\epsilon_n, \|w_n-Az_n\|\leq \epsilon_n;\\
\left\langle u_n, z_n-x\right\rangle+\varphi(z_n,p_n)-\varphi(x,p_n)\leq\epsilon_n \; \; \forall x\in C;\\
\left\langle v_n, w_n-y\right\rangle+\psi(w_n,p_n)-\psi(y,p_n)\leq\epsilon_n \; \; \forall y\in Q.
\end{cases}
\end{equation*}
Let $\delta_n=\|p_n-p\|$. Then $p_n\in\mathbb{B}(p,\delta_n)$ and $\delta_n\rightarrow 0$ as $n\rightarrow\infty$. Thus $(z_n,w_n)\in S_p(\delta_n, \epsilon_n)$.
Let $\epsilon>0$ be such that $S_p(\epsilon,\epsilon)$ is nonempty and bounded. Then there exists a natural number $n_0$ such that
\begin{align}%\label{ineq15}
(z_n,w_n)\in S_p(\delta_n,\epsilon_n)\subset S_p(\epsilon,\epsilon) \; \; \forall n\geq n_0.\nonumber
\end{align}
This implies that the sequence $\{(z_n,w_n)\}$ is bounded. Let $\{(z_{n_k},w_{n_k})\}$ be any subsequence of $\{(z_n,w_n)\}$. Since $\{(z_n,w_n)\}$ is bounded, $\{(z_{n_k},w_{n_k})\}$ has a convergent subsequence. Without loss of generality, we may assume that $(z_{n_k},w_{n_k})\rightarrow (\tilde{z},\tilde{w})$ as $k\rightarrow\infty$. Also, it is clear that $(z_{n_k},w_{n_k})\in S_p(\delta_{n_k}, \epsilon_{n_k})$. Therefore, there exist a sequence $\{(u_{n_k},v_{n_k})\}\subset\mathscr{H}_1\times\mathscr{H}_2$ with $u_{n_k}\in\mathscr{B}_1(z_{n_k},p_{n_k})$, $v_{n_k}\in\mathscr{B}_2(w_{n_k},p_{n_k})$ and a positive sequence $\{\epsilon_{n_k}\}$ of real numbers with $\epsilon_{n_k}\rightarrow 0$ such that
\begin{equation}\label{AS6}
\begin{cases}
d(z_{n_k},C)\leq\epsilon_{n_k}, d(w_{n_k},Q)\leq\epsilon_{n_k}, \|w_{n_k}-Az_{n_k}\|\leq \epsilon_{n_k};\\
\left\langle u_{n_k}, z_{n_k}-x\right\rangle+\varphi(z_{n_k},p_{n_k})-\varphi(x,p_{n_k})\leq\epsilon_{n_k} \; \; \forall x\in C;\\
\left\langle v_{n_k}, w_{n_k}-y\right\rangle+\psi(w_{n_k},p_{n_k})-\psi(y,p_{n_k})\leq\epsilon_{n_k} \; \; \forall y\in Q.
\end{cases}
\end{equation}

Since $\mathscr{B}_1(z_{n_k},p_{n_k})$ and $\mathscr{B}_1(\tilde{z},p)$ are nonempty compact, $\mathscr{B}_1(z_{n_k},p_{n_k})$, $\mathscr{B}_1(\tilde{z},p)\in CB(\mathscr{H}_1)$ for all $k\in\mathbb{N}$. Therefore, by Theorem \ref{Nadler Theorem}, for each $u_{n_k}\in\mathscr{B}_1(z_{n_k},p_{n_k})$, there exists $\tilde{v}_{n_k}\in\mathscr{B}_1(\tilde{z},p)$ such that
\begin{align*}
\|u_{n_k}-\tilde{v}_{n_k}\|\leq (1+1/k) H(\mathscr{B}_1(z_{n_k},p_{n_k}),\mathscr{B}_1(\tilde{z},p))\rightarrow 0 \text{ as } k\rightarrow\infty.
\end{align*}

This implies that $u_{n_k}-\tilde{v}_{n_k}\rightarrow 0$ as $k\rightarrow\infty$. Since $\mathscr{B}_1(\tilde{z},p)$ is compact, without loss of generality, we may assume that $\tilde{v}_{n_k}\rightarrow\tilde{u}\in\mathscr{B}_1(\tilde{z},p)$. This implies that  $u_{n_k}\rightarrow\tilde{u}\in\mathscr{B}_1(\tilde{z},p)$ as $ k\rightarrow\infty$. Similarly, using the assumptions on $\mathscr{B}_2$ we can find $\tilde{v}$ such that $v_{n_k}\rightarrow\tilde{v}\in\mathscr{B}_2(\tilde{w},p)$.

Taking the limit as $k\rightarrow\infty$ in \eqref{AS6} and using our assumptions, we obtain
\begin{equation*}%\label{GPSMVI9}
\begin{cases}
\tilde{z}\in C, \tilde{w}\in Q, \tilde{w}=A\tilde{z};\\
\left\langle\tilde{u}, \tilde{z}-x\right\rangle+\varphi(\tilde{z},p)-\varphi(x,p)\leq 0 \; \; \forall x\in C;\\
\left\langle\tilde{v}, \tilde{w}-y\right\rangle+\psi(\tilde{w},p)-\psi(y,p)\leq 0 \; \; \forall y\in Q.
\end{cases}
\end{equation*}

This implies that $(\tilde{z},\tilde{w})\in S_p$. Since for each $p\in X$, the \eqref{PSMVI} has a unique solution, $(x^*,y^*)=(\tilde{z},\tilde{w})$.
Therefore, every generalized approximate sequence for the \eqref{PSMVI} corresponding to $\{p_n\}$ converges to $(\tilde{z},\tilde{w})$.
Consequently, the $\mathcal{PSMVIP}$ is LP well-posed and the proof is complete.
\end{proof}

\begin{example}\label{Example:3}
Let $\mathscr{H}_1=\mathbb{R}=\mathscr{H}_2$, and let $C=[0,1]=Q$ and $X=\mathbb{R}$. Let $A:\mathbb{R}\rightarrow\mathbb{R}$ be the bounded linear operator defined by $A(x)=x$. We define two functions, $\varphi:\mathbb{R}^2\rightarrow\mathbb{R}$ and $\psi:\mathbb{R}^2\rightarrow\mathbb{R}$, by $\varphi(x,p)=x^4-p^4$ and $\psi(y,p)=y^2-p^2$. It is clear that both $\varphi$ and $\psi$ are continuous functions on $\mathbb{R}^2$. Define two strict multivalued mappings, $\mathscr{B}_1:\mathbb{R}^2\rightarrow 2^\mathbb{R}$ and $\mathscr{B}_2:\mathbb{R}^2\rightarrow 2^\mathbb{R}$, by $\mathscr{B}_1(x,p)=\{x-p,0\}$ and $\mathscr{B}_2(y,p)=\{y-p,0\}$. It is clear that $\mathscr{B}_1$ and $\mathscr{B}_2$ are compact-valued.

Now, we have
\begin{align*}
D(\mathscr{B}_1(x,p),\mathscr{B}_1(y,q))&=\max\{d(x-p,\mathscr{B}_1(y,q)),d(0,\mathscr{B}_1(y,q))\}=d(x-p,\mathscr{B}_1(y,q))\\&
=\min\{\vert x-p\vert,\vert x-y-p+q\vert\}\leq \vert x-y\vert+\vert p-q\vert.
\end{align*}

Since a similar computation holds for $D(\mathscr{B}_1(y,q),\mathscr{B}_1(x,p))$, it follows that
\begin{align*}
H(\mathscr{B}_1(x,p),\mathscr{B}_1(y,q))=\max\{D(\mathscr{B}_1(x,p),\mathscr{B}_1(y,q)),D(\mathscr{B}_1(y,q),\mathscr{B}_1(x,p))\}\leq \vert x-y\vert+\vert p-q\vert.
\end{align*}
This shows that $\mathscr{B}_1$ is $H$-continuous on $\mathbb{R}^2$.

Similarly, we can show that $\mathscr{B}_2$ is also $H$-continuous on $\mathbb{R}^2$.

In addition, it is not difficult to check that $S_p=\{(0,0)\}\subset\mathbb{R}^2$ is a singleton for each $p\in\mathbb{R}$.

Now, for each $p\in\mathbb{R}$ and $\delta, \epsilon> 0$,  we have
\begin{multline}
S_p(\delta,\epsilon) = \bigcup_{q\in[p-\delta,p+\delta]}\Big\{(z, w)\in\mathbb{R}^2: d(z,[0,1])\leq\epsilon, d(w,[0,1])\leq\epsilon, \vert w-z\vert\leq\epsilon; \nonumber\\
\text{ there exist } u\in\mathscr{B}_1(z,q), v\in\mathscr{B}_2(w,q) \text{ such that }\\
\left\langle u, z-x\right\rangle+z^4-x^4\leq\epsilon\;\; \forall x\in [0,1];\\
\left\langle v, w-y\right\rangle+w^2-y^2\leq\epsilon\;\;\forall y\in [0,1]\Big\}.
\end{multline}

One can check that $S_p(\delta,\epsilon)\subset [-\epsilon,1+\epsilon]\times[-\epsilon,1+\epsilon]$ for each $p\in\mathbb{R}$ and each $\delta,\epsilon>0$. Therefore, $S_p(\epsilon,\epsilon)$ is a bounded subset of $\mathbb{R}^2$ for each $p\in\mathbb{R}$. Also, for each $p\in\mathbb{R}$, $S_p\subset S_p(\delta,\epsilon)$ for any $\delta,\epsilon>0$ and consequently, $S_p(\epsilon,\epsilon)\neq\emptyset$. Therefore, the corresponding $\mathcal{PSMVIP}$ associated to these data is LP well-posed by Theorem \ref{Thm17}.
\end{example}

\begin{theorem}\label{Thm18}
Let $\mathscr{H}_1$ and $\mathscr{H}_2$ be two finite-dimensional real Hilbert spaces, and let $C$ and $Q$ be two nonempty, closed and convex subsets of $\mathscr{H}_1$ and $\mathscr{H}_2$, respectively. Let $X$ be a Banach space, and let $\varphi:\mathscr{H}_1\times X\rightarrow\mathbb{R}$ and $\psi:\mathscr{H}_2\times X\rightarrow\mathbb{R}$ be two continuous functions. Let $\mathscr{B}_1:\mathscr{H}_1\times X\rightarrow 2^{\mathscr{H}_1}$ and $\mathscr{B}_2:\mathscr{H}_2\times X\rightarrow 2^{\mathscr{H}_2}$ be two strict compact-valued and $H$-continuous mappings, and let $A:\mathscr{H}_1\rightarrow\mathscr{H}_2$ be a bounded linear operator. Suppose that for each $p\in X$, there exists $\epsilon>0$ such that $S_p(\epsilon,\epsilon)$ is nonempty and bounded. Then the $\mathcal{PSMVIP}$ is LP well-posed in the generalized sense if and only if for each $p\in X$, the \eqref{PSMVI} has a nonempty solution set $S_p$.
\end{theorem}

\begin{proof}
The necessity part is obvious. To prove the sufficiency part, assume that for each $p\in X$, the \eqref{PSMVI} has nonempty solution set $S_p$. Fix $p\in X$ and let $\{p_n\}$ be any sequence in $X$ such that $p_n\rightarrow p$ as $n\rightarrow\infty$. Let $\{(z_n, w_n)\}$ be a generalized approximating sequence for the \eqref{PSMVI} corresponding to $\{p_n\}$. Then there exist a sequence $\{(u_n,v_n)\}\subset\mathscr{H}_1\times\mathscr{H}_2$ with $u_n\in\mathscr{B}_1(z_n,p_n)$, $v_n\in\mathscr{B}_2(w_n,p_n)$ and a positive sequence $\{\epsilon_n\}$ of real numbers with $\epsilon_n\rightarrow 0$ such that
\begin{equation*}%\label{AS7}
\begin{cases}
d(z_n,C)\leq\epsilon_n, d(w_n,Q)\leq\epsilon_n, \|w_n-Az_n\|\leq \epsilon_n;\\
\left\langle u_n, z_n-x\right\rangle+\varphi(z_n,p_n)-\varphi(x,p_n)\leq\epsilon_n \; \; \forall x\in C;\\
\left\langle v_n, w_n-y\right\rangle+\psi(w_n,p_n)-\psi(y,p_n)\leq\epsilon_n \; \; \forall y\in Q.
\end{cases}
\end{equation*}
Let $\delta_n=\|p_n-p\|$. Then $p_n\in\mathbb{B}(p,\delta_n)$ and $\delta_n\rightarrow 0$ as $n\rightarrow\infty$. Thus $(z_n,w_n)\in S_p(\delta_n, \epsilon_n)$.
Let $\epsilon>0$ be such that $S_p(\epsilon,\epsilon)$ is nonempty and bounded. Then there exists a natural number $n_0$ such that
\begin{align}%\label{ineq16}
(z_n,w_n)\in S_p(\delta_n,\epsilon_n)\subset S_p(\epsilon,\epsilon) \; \; \forall n\geq n_0.\nonumber
\end{align}
This implies that $\{(z_n,w_n)\}$ is bounded. Let $\{(z_{n_k},w_{n_k})\}$ be any subsequence of $\{(z_n,w_n)\}$. Since $\{(z_n,w_n)\}$ is bounded, $\{(z_{n_k},w_{n_k})\}$ has a convergent subsequence. Without loss of generality, we may assume that $(z_{n_k},w_{n_k})\rightarrow (\tilde{z},\tilde{w})$ as $k\rightarrow\infty$. Also, it is clear that $(z_{n_k},w_{n_k})\in S_p(\delta_{n_k}, \epsilon_{n_k})$. Therefore, there exist a sequence $\{(u_{n_k},v_{n_k})\}\subset\mathscr{H}_1\times\mathscr{H}_2$ with $u_{n_k}\in\mathscr{B}_1(z_{n_k},p_{n_k})$, $v_{n_k}\in\mathscr{B}_2(w_{n_k},p_{n_k})$ and a positive sequence $\{\epsilon_{n_k}\}$ of real numbers with $\epsilon_{n_k}\rightarrow 0$ such that
\begin{equation}\label{AS8}
\begin{cases}
d(z_{n_k},C)\leq\epsilon_{n_k}, d(w_{n_k},Q)\leq\epsilon_{n_k}, \|w_{n_k}-Az_{n_k}\|\leq \epsilon_{n_k};\\
\left\langle u_{n_k}, z_{n_k}-x\right\rangle+\varphi(z_{n_k},p_{n_k})-\varphi(x,p_{n_k})\leq\epsilon_{n_k} \; \; \forall x\in C;\\
\left\langle v_{n_k}, w_{n_k}-y\right\rangle+\psi(w_{n_k},p_{n_k})-\psi(y,p_{n_k})\leq\epsilon_{n_k} \; \; \forall y\in Q.
\end{cases}
\end{equation}

Since $\mathscr{B}_1(z_{n_k},p_{n_k})$ and $\mathscr{B}_1(\tilde{z},p)$ are nonempty compact, $\mathscr{B}_1(z_{n_k},p_{n_k}), \mathscr{B}_1(\tilde{z},p)\in CB(\mathscr{H}_1)$ for all $k\in\mathbb{N}$. Therefore, by Theorem \ref{Nadler Theorem}, for each $u_{n_k}\in\mathscr{B}_1(z_{n_k},p_{n_k})$, there exists $\tilde{v}_{n_k}\in\mathscr{B}_1(\tilde{z},p)$
such that
\begin{align*}
\|u_{n_k}-\tilde{v}_{n_k}\|\leq (1+1/k) H(\mathscr{B}_1(z_{n_k},p_{n_k}),\mathscr{B}_1(\tilde{z},p))\rightarrow 0 \text{ as } k\rightarrow\infty.
\end{align*}

This implies that $u_{n_k}-\tilde{v}_{n_k}\rightarrow 0$ as $k\rightarrow\infty$. Since  $\mathscr{B}_1(\tilde{z},p)$ is compact, without loss of generality, we may assume that $\tilde{v}_{n_k}\rightarrow\tilde{u}\in\mathscr{B}_1(\tilde{z},p)$. This implies that $u_{n_k}\rightarrow\tilde{u}\in\mathscr{B}_1(\tilde{z},p)$ as $ k\rightarrow\infty$.
Similarly, using the assumptions on $\mathscr{B}_2$ we can find $\tilde{v}$ such that $v_{n_k}\rightarrow\tilde{v}\in\mathscr{B}_2(\tilde{w},p)$.

Taking the limit as $k\rightarrow\infty$ in \eqref{AS8} and using our assumptions, we obtain
\begin{equation}\label{GPSMVI9}
\begin{cases}
\tilde{z}\in C, \tilde{w}\in Q, \tilde{w}=A\tilde{z};\\
\left\langle\tilde{u}, \tilde{z}-x\right\rangle+\varphi(\tilde{z},p)-\varphi(x,p)\leq 0 \; \; \forall x\in C;\\
\left\langle\tilde{v}, \tilde{w}-y\right\rangle+\psi(\tilde{w},p)-\psi(y,p)\leq 0 \; \; \forall y\in Q.
\end{cases}
\end{equation}

This implies that $(\tilde{z},\tilde{w})\in S_p$. Therefore, each generalized approximate sequence for the \eqref{PSMVI} corresponding to $\{p_n\}$ converges to $(\tilde{z},\tilde{w})$.
Consequently, the $\mathcal{PSMVIP}$ is  LP well-posed in the generalized sense and the proof is complete.
\end{proof}

\begin{example}\label{Example:4}
Let $\mathscr{H}_1=\mathbb{R}=\mathscr{H}_2$, and let $C=[-1,1]=Q$ and $X=\mathbb{R}$. Let $A:\mathbb{R}\rightarrow\mathbb{R}$ be the bounded linear operator defined by $A(x)=x$. We define two functions, $\varphi:\mathbb{R}^2\rightarrow\mathbb{R}$ and $\psi:\mathbb{R}^2\rightarrow\mathbb{R}$, by $\varphi(x,p)=(x^2-1)^2-p^2$ and $\psi(y,p)=(y^4-1)^2-p^4$. It is clear that $\varphi$ and $\psi$ are both continuous functions on $\mathbb{R}^2$. Define two strict multivalued mappings, $\mathscr{B}_1:\mathbb{R}^2\rightarrow 2^\mathbb{R}$ and $\mathscr{B}_2:\mathbb{R}^2\rightarrow 2^\mathbb{R}$, by $\mathscr{B}_1(x,p)=\{x-p,0\}$ and $\mathscr{B}_2(y,p)=\{y-p,0\}$. It is clear that $\mathscr{B}_1$ and $\mathscr{B}_2$ are compact-valued.

Now, we have
\begin{align*}
D(\mathscr{B}_1(x,p),\mathscr{B}_1(y,q))&=\max\{d(x-p,\mathscr{B}_1(y,q)),d(0,\mathscr{B}_1(y,q))\}=d(x-p,\mathscr{B}_1(y,q))\\&
=\min\{\vert x-p\vert,\vert x-y-p+q\vert\}\leq \vert x-y\vert+\vert p-q\vert.
\end{align*}

Since a similar computation holds for $D(\mathscr{B}_1(y,q),\mathscr{B}_1(x,p))$, it follows that
\begin{align*}
H(\mathscr{B}_1(x,p),\mathscr{B}_1(y,q))=\max\{D(\mathscr{B}_1(x,p),\mathscr{B}_1(y,q)),D(\mathscr{B}_1(y,q),\mathscr{B}_1(x,p))\}\leq \vert x-y\vert+\vert p-q\vert.
\end{align*}

This shows that $\mathscr{B}_1$ is $H$-continuous on $\mathbb{R}^2$.

Also, it is not difficult to check that $(-1,-1), (1,1)\in S_p\subset\mathbb{R}^2$ for each $p\in\mathbb{R}$.

Now, for each $p\in\mathbb{R}$ and $\delta, \epsilon> 0$,  we have
\begin{multline}
S_p(\delta,\epsilon) = \bigcup_{q\in[p-\delta,p+\delta]}\Big\{(z, w)\in\mathbb{R}\times\mathbb{R}: d(z,[-1,1])\leq\epsilon, d(w,[-1,1])\leq\epsilon, \vert w-z\vert\leq\epsilon; \nonumber\\
\text{ there exist } u\in\mathscr{B}_1(z,q), v\in\mathscr{B}_2(w,q) \text{ such that }\\
\left\langle u, z-x\right\rangle+(z^2-1)^2-(x^2-1)^2\leq\epsilon\;\;\forall x\in [-1,1];\\
\left\langle v, w-y\right\rangle+(w^4-1)^2-(y^4-1)^2\leq\epsilon\;\;\forall y\in [-1,1]\Big\}.
\end{multline}

One can check that $S_p(\delta,\epsilon)\subset [-1-\epsilon,1+\epsilon]\times[-1-\epsilon,1+\epsilon]$ for each $p\in\mathbb{R}$ and each $\delta, \epsilon>0$. Therefore, $S_p(\epsilon,\epsilon)$ is a bounded subset of $\mathbb{R}^2$. Also, for each $p\in\mathbb{R}$, $S_p\subset S_p(\delta,\epsilon)$ for any $\delta,\epsilon>0$ and consequently, $S_p(\epsilon,\epsilon)\neq\emptyset$. Therefore, the corresponding $\mathcal{PSMVIP}$ associated to these data is LP well-posed in the generalized sense by Theorem \ref{Thm18}.
\end{example}

%=======================================================================================================
\section{Important special cases of the \eqref{SMVI}}
\label{Sec:5}

\begin{enumerate}
\item{\bf Split feasibility problem}: In particular, if we take $\mathscr{B}_1=\{0\}, \mathscr{B}_2=\{0\}, f=0=g$, then the \eqref{SMVI} reduces to the following split feasibility problem (SFP): find a point $x^*\in\mathscr{H}_1$ such that
\begin{align}\label{SFP}
x^*\in C \text{ and } Ax^*\in Q.\tag{SFP}
\end{align}

The \eqref{SFP} was first introduced and studied by Censor and Elfving \cite{YCEN1994} in 1994.

\item{\bf Split variational inequality problem}:  In particular, if we take $\mathscr{B}_1=F_1, \mathscr{B}_2=F_2$, where $F_1:\mathscr{H}_1\rightarrow\mathscr{H}_1$ and $F_2:\mathscr{H}_2\rightarrow\mathscr{H}_2$ are two single-valued mappings and $f=0=g$, then the \eqref{SMVI} reduces to the following split variational inequality problem (SVIP): find a point $(x^*,y^*)\in \mathscr{H}_1\times\mathscr{H}_2$ such that
\begin{equation}\label{SVIP}
\begin{cases}
x^*\in C, y^*\in Q, y^*=Ax^*;\\
\left\langle F_1(x^*), x-x^*\right\rangle\geq 0 \; \; \forall x\in C;\\
\left\langle F_2(y^*), y-y^*\right\rangle\geq 0 \; \; \forall y\in Q.\tag{SVIP}
\end{cases}
\end{equation}

The \eqref{SVIP} was first introduced and studied by Censor et al. \cite{YCEN22012} in 2012.

\item{\bf Split mixed variational inequality problem}:  In particular, if we take $\mathscr{B}_1=F_1, \mathscr{B}_2=F_2$, where $F_1:\mathscr{H}_1\rightarrow\mathscr{H}_1$ and $F_2:\mathscr{H}_2\rightarrow\mathscr{H}_2$ are two single-valued mappings, then the \eqref{SMVI} reduces to the following split mixed variational inequality problem (SMVIP): find a point $(x^*,y^*)\in \mathscr{H}_1\times\mathscr{H}_2$ such that
\begin{equation}\label{SMVIP}
\begin{cases}
x^*\in C, y^*\in Q, y^*=Ax^*;\\
\left\langle F_1(x^*), x-x^*\right\rangle+f(x)-f(x^*)\geq 0 \; \; \forall x\in C; \\
\left\langle F_2(y^*), y-y^*\right\rangle+g(y)-g(y^*)\geq 0 \; \; \forall y\in Q. \tag{SMVIP}
\end{cases}
\end{equation}

\item{\bf Split minimization problem}: In particular, if we take $\mathscr{B}_1=\{0\}, \mathscr{B}_2=\{0\}$, then the \eqref{SMVI} reduces to the following split minimization problem (SMP): find a point $(x^*,y^*)\in \mathscr{H}_1\times\mathscr{H}_2$ such that
\begin{equation}\label{SMP}
\begin{cases}
x^*\in C, y^*\in Q, y^*=Ax^*;\\
f(x^*)=\min_{x\in C} f(x);\\
g(y^*)=\min_{y\in Q} g(y).\tag{SMP}
\end{cases}
\end{equation}
\end{enumerate}

\text{}

\text{}

{\bf Acknowledments}. Soumitra Dey acknowledges the financial support of the Post-Doctoral Program at the Technion - Israel Institute of Technology. Simeon Reich was partially supported by the Israel Science Foundation (Grant 820/17), by the Fund for the Promotion of Research at the Technion (Grant 2001893) and by the Technion General Research Fund (Grant 2016723).\\

\text{}

\section*{Disclosure statement}
No potential conflict of interest was reported by the authors.

\section*{Data availability statement}
The authors acknowledge that the data presented in this study must be deposited and made
publicly available in an acceptable repository, prior to publication.

%=======================================================================================================


\begin{thebibliography}{99}

\bibitem{ANTY1966} Tykhonov, A.N.: On the stability of the functional optimization problem. USSR J. Comput. Math. Math. Phys. {\bf 6}, 631--634 (1966)

\bibitem{MFUR1970} Furi, M., Vignoli, A.: About well-posed optimization problems for functionals in metric spaces. J. Optim. Theory Appl. {\bf 5}(3), 225--229 (1970)

\bibitem{ESLE1966} Levitin, E.S., Polyak, B.T.: Convergence of minimizing sequences in conditional extremum problems. Sov. Math. Dokl. {\bf 7}, 764--767 (1966)

\bibitem{TZOL1995} Zolezzi, T.: Well-posedness criteria in optimization with application to the calculus of variations. Nonlinear Anal. TMA {\bf 25}, 437--453 (1995)

\bibitem{TZOL1996} Zolezzi, T.: Extended well-posedness of optimization problems. J. Optim. Theory Appl. {\bf 91}, 257--266 (1996)

\bibitem{XXHU2006} Huang, X.X., Yang, X.Q.: Generalized Levitin-Polyak well-posedness in constrained optimization. SIAM J. Optim. {\bf 17}, 243--258 (2006)

\bibitem{ALDO1993} Dontchev, A.L., Zolezzi, T.: Well-posed optimization problems. Springer, Berlin (1993)

\bibitem{DKIN1980} Kinderlehrer, D., Stampacchia, G.: An introduction to variational inequalities and their applications. SIAM Academic Press, New York (1980)

\bibitem{GSTA1964} Stampacchia, G.: Formes bilineaires coercitives sur les ensembles convexes. Comp. Rend. Hebd. s\'eances l'Acad. Sci. {\bf 258}, 4413--4416 (1964)

\bibitem{PHAR1966} Hartman, P., Stampacchia, G.: On some nonlinear elliptic differential functional equations. Acta Math. {\bf 115}, 271--310 (1966)

\bibitem{RLUC1981} Lucchetti, R., Patrone, F.: A characterization of Tyhonov well-posedness for minimum problems, with applications to variational inequalities. Numer. Funct. Anal. Optim. {\bf 3}(4), 461-476 (1981)

\bibitem{RLUC1982} Lucchetti, R., Patrone, F.: Some properties of well-posed variational inequalities governed by linear operators. Numer. Funct. Anal. Optim. {\bf 5}(3),
349--361 (1982/83)

\bibitem{RHU2010} Hu, R., Fang, Y.P.: Levitin-Polyak well-posedness of variational inequalities. Nonlinear Anal. {\bf 72}, 373--381 (2010)

\bibitem{XXHU2007} Huang, X.X., Yang, X.Q.: Levitin-Polyak well-posedness in generalized variational inequality problems with functional constraints. J. Ind. Manag. Optim.
{\bf 3}, 671--684 (2007)

\bibitem{XXHU2009} Huang, X.X., Yang, X.Q., Zhu, D.L.: Levitin-Polyak well-posedness of variational inequality problems with functional constraints. J. Glob. Optim. {\bf 44}(2), 159--174 (2009)

\bibitem{YPFA2010} Fang, Y.P., Huang, N.J., Yao, J.C.: Well-posedness by perturbations of mixed variational inequalities in Banach spaces. Eur. J. Oper. Res. {\bf 201}(3), 682--692 (2010)

\bibitem{YCEN2011} Censor, Y., Gibali, A., Reich, S.: Strong convergence of subgradient extragradient methods for the variational inequality problem in Hilbert space.
Optim. Methods Softw. {\bf 26}, 827--845 (2011)

\bibitem{YCEN22011}  Censor, Y., Gibali, A., Reich, S.: The subgradient extragradient method for solving variational inequalities in Hilbert space. J. Optim. Theory Appl.
{\bf 148}, 318--335 (2011)

\bibitem{XJCA2014} Cai, X.J., Gu, G.Y., He,  B.S.: On the $O(\frac{1}{t})$ convergence rate of the projection and contraction methods for variational inequalities with Lipschitz continuous monotone operators. Comput Optim Appl. {\bf 57}, 339--363 (2014)

\bibitem{QLDO2019} Dong, Q.L., Yang, J.F., Yuan, H.B.: The projection and contraction algorithm for solving variational inequality problems in Hilbert space. J Nonlinear Convex Anal. {\bf 20}(1), 111--122 (2019)

\bibitem{QLDO2018} Dong, Q.L., Cho, Y.J., Zhong, L.L., Rassias, M.: Inertial projection and contraction algorithms for variational inequalities. J Global Optim. {\bf 70}, 687--704 (2018)

\bibitem{SCFA1982} Fang, S.C., Peterson, E.L.: Generalized Variational Inequalities. J. Optim. Theory Appl. {\bf 38}, 363--383 (1982)

\bibitem{EBLU1994} Blum, E., Oettli, W.: From optimization and variational inequalities to equilibrium problems. Math. Student {\bf 63}, 123--145 (1994)

\bibitem{RSAI1976} Saigal, R.: Extension of the generalized complementarity problem. Math. Oper. Res. {\bf 1}, 260-266 (1976)

\bibitem{SDEY22023} Dey, S., Vetrivel, V., Xu, H.K.: Well-posedness for the split equilibrium problem, Optim. Lett. (2023), 13 pp
https://doi.org/10.1007/s11590-023-02034-4

\bibitem{KQWU2007}  Wu, K.Q., Huang, N.J.: Properties of the generalized $f$-projection operator and its applications in Banach spaces. Comput. Math. Appl. {\bf 54}, 399--406 (2007)

\bibitem{YPFA2008} Fang, Y.P., Huang, N.J., Yao, J.C.: Well-posedness of mixed variational inequalities, inclusion problems and fixed point problems. J. Global Optim. {\bf 41}, 117--133 (2008)

\bibitem{FEBR1965} Browder, F.E.: Multivalued monotone nonlinear mappings and duality mappings in Banach spaces. Trans. Amer. Math. Soc. {\bf 18}, 338--351 (1965)

\bibitem{TVTH2022} Thang, T.V., Hien, N.D., Thach, H.T.C., Anh, P.N.: Weak convergence of inertial proximal algorithms with self-adaptive stepsize for solving multivalued variational inequalities. Optimization (2022), 29 pp.

\bibitem{LCCE2008} Ceng, L.C., Yao, J.C.: Well-posedness of generalized mixed variational inequalities, inclusion problems and fixed-point problems. Nonlinear Anal. {\bf 69}, 4585--4603 (2008)

\bibitem{YCEN22012} Censor, Y., Gibali, A., Reich, S.: Algorithms for the Split Variational Inequality Problem. Numer. Algorithms {\bf 59}, 301--323 (2012)

\bibitem{AMOU2011} Moudafi, A.: Split Monotone Variational Inclusion. J. Optim. Theory Appl. {\bf 150}, 275--283 (2011)

\bibitem{YCEN1994} Censor, Y., Elfving, T.: A multiprojection algorithm using Bregman projection in a product space. Numer. Algorithms {\bf 8}, 221--239 (1994)

\bibitem{SDEY2023} Dey, S., Gibali, A., Reich, S.: Levitin-Polyak well-posedness for split equilibrium problems. Rev. R. Acad. Cienc. Exactas F\'{\i}s. Nat. Ser. A Mat. RACSAM 117 (2023), no. 2, Paper No. 88, 18 pp.

\bibitem{DVHI2022} Hieu, D.V., Reich, S., Anh, P.K., Ha, N.H.: A new proximal-like algorithm for solving split variational inclusion problems. Numer. Algorithms {\bf 89}, 811--837 (2022)


\bibitem{RHUA2016} Hu, R., Fang, Y.P.: Characterizations of Levitin-Polyak well-posedness by perturbations for the split variational inequality problem. Optimization {\bf 65}(9), 1717--1732 (2016)

\bibitem{UMOS1967} Mosco, U.: A remark on a theorem of F. E. Browder$^*$. J. Math. Anal. Appl. {\bf 20}, 90--93 (1967)

\bibitem{HHBA2011} Bauschke, H.H., Combettes, P.L.: Convex analysis and monotone operator theory in Hilbert spaces. Springer, New York (2011)

\bibitem{CHOR1985} Horvath, C.: Measure of Non-compactness and Multivalued Mappings in Complete Metric Topological Vector Spaces. J. Math. Anal. Appl. {\bf 108}, 403--408 (1985)

\bibitem{KKUR1968} Kuratowski, K.: Topology. Vol. I+II, Academic Press, New York (1968)

\end{thebibliography}
%%% if required, the content of .bbl file can be included here once bbl is generated
%%%\input sn-article.bbl

\end{document}